\documentclass[12pt, reqno]{amsart}
\usepackage{amsmath, amsthm, amscd, amsfonts, amssymb, graphicx, color}
\usepackage[bookmarksnumbered, colorlinks, plainpages]{hyperref}

\newtheorem{thm}{Theorem}[section]
 \newtheorem{cor}[thm]{Corollary}
 \newtheorem{lem}[thm]{Lemma}
 \newtheorem{prop}[thm]{Proposition}
 \theoremstyle{definition}
 \newtheorem{defn}[thm]{Definition}
 \theoremstyle{remark}
 \newtheorem{rem}[thm]{Remark}
 
 \numberwithin{equation}{section}
\begin{document}
\setcounter{page}{1}

\title[Bergman kernel estimates and Toeplitz operators]{Bergman kernel estimates and Toeplitz \\ operators on holomorphic line bundles}
\author[Said Asserda]{Said Asserda}
\address{Ibn tofail University, Faculty of Sciences, Department of Mathematics, PO 242 Kenitra Morocco.}
\email{asserda-said@univ-ibntofail.ac.ma}
\date{\today}
\begin{abstract}
We characterize operator-theoretic properties
(boundedness, compactness, and Schatten class membership) of Toeplitz
operators with positive measure symbols on Bergman spaces of   holomorphic hermitian line bundles  over  K\"ahler Cartan-Hadamard manifolds in terms of geometric or operator-theoretic properties of measures.
\end{abstract} \maketitle
\section{Introduction and statment of results}
The purpose of this paper is to extend the standard theory dealing with
boundedness, compactness, and Schatten class membership of Toeplitz operators
with nonnegative measure symbols on  generalized Bargman-Fock spaces [6,12,14,15,16,21,22,25,30] to Bergman spaces  of holomorphic sections of hermitian holomorphic line bundles over  K\"ahler Cartan-Hadamard manifolds. As an application, we give a characterization  of self-holomorphic maps whose composition operators  bounded, compact or belongs to the Schatten ideal class which extend  previous results for generalized Bargman-Fock spaces [4,27,28,29,34]. \\
Let $(M,g)$ be a complex  hermitian manifold and $(L,h)\longrightarrow M$ be an  holomorphic hermitian line bundle. For $p\in ]0,\infty]$, define $\mathcal{F}^{p}(M,L)$ the $\mathbb{C}$-vector space of holomorphic sections $ s : M\longrightarrow L$ such that
$$
\Vert s\Vert_{2}:=\Bigl(\int_{M}\vert s\vert_{h}^{p}dv_{g}\Bigr)^{1\over p} < \infty
$$
Let  $P$ the orthogonal projection from the Hilbert space of $L^{2}(M,L)$ onto its closed subspace $\mathcal{F}^{2}(M,L)$. Let  $K\in C^{\infty}(M\times M,L\otimes \overline{L})$  the reproducing ( or Bergman ) kernel of $P$, that is
$$
K(z,w)=\sum_{j=1}^{d}s_{j}(z)\otimes\overline{ s_{j}(w)}
$$
where $\overline{L}$  is the conjugate bundle of $L$,  $(s_{j})$ is an orthonormal basis for $\mathcal{F}^{2}(M,L)$ and $ d=\hbox{\sl dim}\mathcal{F}^{2}(M,L)\leq\infty$. \\ \\
The first result of this paper is a pointwise  estimate for the Bergman kernel of $L$ in spirit of  those obtained in [1,5,20] for $n=1$ and in [8,18,26] for $n\geq 2$.
\begin{thm} Let $(M,g)$ be a  Stein K\"ahler manifold with bounded geometry.
Let $(L,h)\longrightarrow (M,g)$ be a hermitian holomorphic line bundle with bounded curvature such that
$$
c(L)+Ricci(g)\geq a\omega_{g}
$$
for some positive constant $a$. There are constants $\alpha,\ C > 0$ such that for all $z,w \in M$ ,
$$
\vert K(z,w)\vert\leq Ce^{-\alpha d_{g}(z,w)}
$$
\end{thm}
For a positive measure $\mu$,  the  Toeplitz operator $T_{\mu}$ with symbol $\mu$ is defined formaly by
$$
T_{\mu}s(z)=\int_{M} K(z,w)\bullet s(w)d\mu(w)
$$
where $z\rightarrow K(z,w)\bullet s(w) \in L_{z}$ is the holomorphic section of $L$ defined
$$
K(z,w)\bullet s(w):=\sum_{j=1}^{d}<s(w),s_{j}(w)>_{L_{w}}s_{j}(z)
$$
Let $\tilde{\mu} : M\rightarrow \mathbb{R}^{+}$ the Berezin transform of $\mu$ :
$$
\tilde{\mu}(z):=\int_{M}\vert k_{z}(w)\vert^{2}d\mu(w)
$$
where
$$
k_{z}(w):={K(w,z)\over\sqrt{\vert K(z,z)\vert}}
$$
Let  $T : H_{1}\rightarrow H_{2}$ be a compact operator  from two Hilbert spaces $H_{1}$ and $H_{2}$ and
$$
Tf=\sum_{n=0}^{\infty}\lambda_{n}<f,e_{n}>\sigma_{n},\quad f\in H_{1},
$$
 its Schmidt decomposition where $(e_{n})$  ( resp. $(\sigma_{n})$)  is an orthonormal basis of $H_{1}$ ( resp. $H_{2}$  ) and
$(\lambda_{n})$ a sequence with $\lambda_{n} > 0$ and $\lambda_{n}\rightarrow 0$ ( see [30]).
For $0 < p \leq \infty$, the compact operator $T$ belongs to the Schatten-von Neumann $p$-class $\mathcal{S}_{p}(H_{1},H_{2})$ if and only if
$$
\Vert T\Vert_{\mathcal{S}_{p}}^{p}:=\sum_{n=0}^{\infty}\lambda_{n}^{p} < \infty
$$
Let $(N,\omega_{N})$ be an herimitian manifold. Let $\Phi : N\rightarrow M$ be a holomorphic map and \begin{eqnarray*}
C_{\Phi} : \mathcal{F}^{2}(M,L)&\longrightarrow &\mathcal{F}^{2}(N,\Phi^{*}L)\\
s&\longrightarrow & s\circ\Phi
\end{eqnarray*}
the composition operator associated to $\Phi$. We define the transform $B_{\Phi}$
(related to the usual Berezin transform) associated to $\Phi$ to be a function on $M$ as follows
$$
B_{\Phi}(z)^{2}=\int_{M}\vert K(z,w)\vert^{2}d\nu_{\Phi}(w)
$$
where $\nu_{\Phi}$ is the pull-back measure defined as follows : for all Borel set $E\subset M$
$$
\nu_{\Phi}(E)=\int_{N}\mathsf{1}_{\Phi^{-1}(E)}(w)dv_{\omega_{N}}(w)
$$
Our second result of this paper is the characterization  of operator-theoretic properties (boundedness, compactness, and Schatten class membership)
 of Toeplitz operators with positive measure symbols on Bergman space of holomorpic sections which extend those for generalized Bargman-Fock spaces.
\begin{thm} Let $(M,g)$ be a K\"ahler Cartan-Hadamard manifold  with bounded geometry and uniformly subexponentially volume growth. Let $(L,h)\longrightarrow (M,g)$ be a  holomorphic hermitian line bundle with bounded curvature such that
$$
 c(L)+ Ricci(g)\geq a\omega_{g}
$$
for some positive constant $a$. Let $\mu$ be a positive measure on $M$. Let $p\in [1,+\infty]$. The following  conditions are equivalent \\
(a) The  operator $T_{\mu} : \mathcal{F}^{p}(M,L)\longrightarrow\mathcal{F}^{p}(M,L)$ is bounded $(1\leq p\leq\infty)$.\\
(b) $\mu$ is a  Carleson measure.\\
(c) $ \tilde{\mu}$ is bounded on $M$.\\
(d) There exists $\delta > 0$ such that  the function $z\rightarrow\mu(B_{g}(z,\delta))$ is bounded.
 \end{thm}
\begin{thm} Let $(M,g)$ be a K\"ahler Cartan-Hadamard manifold  with bounded geometry and uniformly subexponentially volume growth. Let $(L,h)\longrightarrow (M,g)$ be a  holomorphic hermitian line bundle with bounded curvature such that
$$
 c(L)+ Ricci(g)\geq a\omega_{g}
$$
for some positive constant $a$. Let $\mu$ be a positive measure on $M$. Let $p\in [1,\infty]$. The following  conditions are equivalent :\\
(a)  The  operator $T_{\mu} : \mathcal{F}^{2}(M,L)\longrightarrow\mathcal{F}^{2}(M,L)$ is compact.\\
(b)  $\mu$ is a vanishing Carleson measure.\\
(c)  $\displaystyle\lim_{d_{g}(z,z_{0})\rightarrow\infty}\tilde{\mu}(z)=0$.\\
(d)  There exists $\delta > 0$ such that  $\displaystyle\lim_{d_{g}(z,z_{0})\rightarrow\infty}\mu(B_{g}(z,\delta))=0$\\
\end{thm}
\begin{thm} Let $(M,g)$ be a K\"ahler Cartan-Hadamard manifold  with bounded geometry and uniformly subexponentially volume growth. Let $(L,h)\longrightarrow (M,g)$ be a  holomorphic hermitian line bundle with bounded curvature such that
$$
 c(L)+ Ricci(g)\geq a\omega_{g}
$$
for some positive constant $a$. Let $\mu$ be a positive measure on $M$. The following  conditions are equivalent :\\
(a)  The operator $T_{\mu} : \mathcal{F}^{2}(M,L)\longrightarrow\mathcal{F}^{2}(M,L)$ belong to $\mathcal{S}_{p}\ ( 0 < p \leq\infty)$.\\
(b)  $\tilde{\mu}\in L^{p}(M,dv_{g})$.\\
(c)  There exists $\delta > 0$ such that the function $z\rightarrow\mu(B_{g}(z,\delta))\in L^{p}(M,dv_{g})$.\\
(d)  There exists $\delta > 0$ and an $r$-lattice  $(a_{j})$
 such that  $\{\mu(B_{g}(a_{j},\delta))\}\in \ell^{p}(\mathbb{N})$.\\
 Morever there is a positive constant $C$ such that
 $$
{1\over C}\Vert \tilde{\mu}\Vert_{L^{p}(M,dv_{g})}\leq\Vert T_{\mu}\Vert_{\mathcal{S}_{p}}\leq C\Vert \tilde{\mu}\Vert_{L^{p}(M,dv_{g})}
$$
\end{thm}
\noindent For boundedness, compactness, and Schatten class membership of composition operators, we have the following result which exend those for Bargman-Fock spaces.
\begin{thm} Let $(M,g)$ be a K\"ahler Cartan-Hadamard manifold  with bounded geometry and uniformly subexponentially volume growth. Let $(L,h)\longrightarrow (M,g)$ be a  holomorphic hermitian line bundle with bounded curvature such that
$$
 c(L)+ Ricci(g)\geq a\omega_{g}
$$
for some positive constant $a$. Let $(N,\omega_{N})$ be an Hermitian manifold. Let $\Phi : M\rightarrow M$ be a holomorphic map and \begin{eqnarray*}
C_{\Phi} : \mathcal{F}^{2}(M,L)&\longrightarrow &\mathcal{F}^{2}(N,\Phi^{*}L)\\
s&\longrightarrow & s\circ\Phi
\end{eqnarray*}
the composition operator associated to $\Phi$. Let $0< p < \infty$. Then \\
(i) $C_{\Phi}$ is bounded if and only if $ \nu_{\Phi}$ is a Carleson measure for
$\mathcal{F}^{2}(M,L)$ if and only if $B_{\Phi}$ is bounded.\\
(ii) $C_{\Phi}$ is compact if and only if $ \nu_{\Phi}$ is a vanishing Carleson measure for $\mathcal{F}^{2}(M,L)$ if and only if $ B_{\Phi}$ vanishes at infinity.\\
(iii) $C_{\Phi}$ is in Schatten $p$-class if and only if $ B_{\Phi}\in L^{p}(M,dv_{g})$. Morever there is a positive constant $C$ such that
$$
{1\over C}\Vert B_{\Phi}\Vert_{L^{p}(M,dv_{g})}\leq\Vert C_{\Phi}\Vert_{\mathcal{S}_{p}}\leq C\Vert B_{\Phi}\Vert_{L^{p}(M,dv_{g})}
$$
\end{thm}
Characterizations of bounded, compact and Schatten class Toeplitz operators with positive measure symbols  on  generalized Bargmann-Fock space  or on weighted Bergman spaces of bounded strongly pseudoconvex domains,  in terms of Carleson measures and the Berezin transform, depend strongly on off-diagonal exponential decay  of the Bergman kernel. In the spirit of [8], we establish a similar   off-diagonal decay of the Bergman kernel associated to holomorphic hermitian line bundles whose curvature is uniformly comparable to the metric form.\\

This paper consists of five sections. In the next section, we will recall
some  definitions and properties  of  K\"ahler manifolds, Bergman Kernel of line bundles, $\bar\partial$-methods and Toeplitz operators. In section 3 we provide useful estimates for Bergman kernel and we prove Theorem 1.1. In section 4 we will prove Theorems 1.2, 1.3. In section 5 we will prove theorems 1.4 and 1.5.
\section{Preliminary}
\subsection{Curvatures in K\"ahlerian Geometry}
Let $(M,J,g)$ be a complex $n$-manifold with a Riemannian metric $g$ which is Hermitian i.e. $$g(JX,JY)=g(X,Y),\  \forall X,Y \in TM\ \hbox{( real tangent vectors )}
$$ and a complex structure $J : TM\rightarrow TM$ i.e $J^{2}=-Id_ {TM}$. Assume furthermore that $g$ is K\"ahler, i.e  the real $2$-form $$\omega_{g}(X,Y)=g(JX,Y)$$ is closed.
In local coordiantes  $z^{1},z^{2},\cdots,z^{n}$ of
$M$
$$
g=\sum_{i,j=1}^{n}g_{i\bar{j}}dz^{i}\otimes d\bar{z^{j}},\quad \omega={\sqrt{-1}\over 2}\sum_{i,j=1}^{n}g_{i\bar{j}}dz^{i}\wedge d\bar{z^{j}}
$$
The coefficients of the curvature tensor $R$ of $g$ are given by
$$
R_{i\bar{j}k\bar{l}}=-{\partial^{2}g_{i\bar{j}}\over\partial z^{k}\partial\bar{z}^{l}}+\sum_{p,q=1}^{n}g^{q\bar{p}}{\partial g_{i\bar{p}}\over\partial z^{k}}{\partial g_{q\bar{j}}\over\partial\bar{z}^{l}}
$$
The sectional curvature of a $2$-plane $\sigma\subset T_{x}M$ is defined as
$$
K(\sigma):=R(X,Y,Y,X)=R(X,Y,JY,JX)
$$
where $\{X,Y\}$ is an orthonormal basis of $\sigma$.
\begin{defn}
We say that $(M,g)$ has non-positive sectional curvature if
$$K(\sigma)\leq 0\quad\hbox{for all}\ 2-\hbox{plane $\sigma\subset TM$}
$$
A Cartan-Hadamad manifold is  a simply conneceted complete manifold with negative sectional curvature. Since each point in a Cartan-Hadamard manifold is a pole then the square of the  distance function at such point is smooth.
 \end{defn}
\noindent The Ricci curvature of $g$ is the $(1,1)$-forme
$$
Ric(g):={i\over 2\pi}\sum_{i,jk,l=1}^{n}g^{k\bar{l}}R_{i\bar{j}k\bar{l}}dz^{i}\wedge d\bar{z}^{j}
$$
In local coordinates
$$
Ric(g)=-{i\over 2\pi}\sum_{i,j=1}^{n}{\partial^{2}\log\det(g_{k\bar{l}})\over\partial z^{i}\partial\bar{z}^{j}}dz^{j}\wedge d\bar{z}^{l}
$$
\begin{defn}
We say  that the Ricci curvature of $(M,g)$ has lower bound $C\in \mathbb{R}$  if
$$
Ric(g)(X,X)\geq C \omega_{g}(X,X)\quad\hbox{for all $X\in T^{(1,0)}M$}
$$
\end{defn}
\noindent Denote by $d_{g}(z,w)$ the Riemannian distance from $z\in X$ to $w\in X$ and $B(z,r)=\{ w\in M\ :\  d_{g}(w,z) < r\}$ the open geodesic ball. The manifold $(M,g)$ is  complete if $(M,d_{g})$ is a complete  metric space.\\  \\
\noindent Given $(M,g)$ a Riemannian manifold, we say that a family $(\Omega_{k})$ of open subsets of $M$ is a uniformly locally finite covering of $M$ if the following holds  \\ $(\Omega_{k})$ is a covering of $M$, and there exits an integer $N$ such that each point $x\in M$ has a neighborhood which intersect at most $N$ of the $\Omega_{k}$. One then has the following Gromov's Packing Lemma [11].
 \begin{lem}
 Let $(M,g)$ be a smooth, compete Riemannian $n$-manifol with Ricci curvature bounded from below by some $K$ real, and let $\rho > 0$ be given. There exists a sequence $(x_{i})$ of points of $M$ such that for every $r\ge\rho$ :\\
 (i) the family $(B_{g}(x_{i},r))$ is a uniformly locally finite covering of $M$, and there is an upper bound for $N$ in terms of $n,r,\rho,$ and $K$\\
 (ii) for any $i\not=j,\ B_{g}(x_{i}{\rho\over 2})\cap B_{g}(x_{j},{\rho\over 2})=\emptyset$
 \end{lem}
 \begin{defn} We say that the volume of $(M,g)$ grows uniformly subexponentially if and only if for any $\epsilon > 0$ there exists a constant $C < \infty$ such that, for ll $r > 0$ and all $z\in M$
 $$
 vol_{g}(B(z,r))\leq Ce^{\epsilon r}vol_{g}(B(z,1))
 $$
\end{defn}
\begin{defn}
A  Hermitian manifold $(M,g)$ is said to have bounded geometry if there exists positive numbers $R$ and  $c$  such that for all $z\in M$ there exists a biholomorphic mapping $F_{z} : (U,0)\subset\mathbb{C}^{n}\longrightarrow (V,z)\subset M$ such that \\
 (i) $F_{z}(0)=z$,\\
 (ii) $B_{g}(z,R)\subset F_{z}(U)$ and \\
 (iii) ${1\over c} g_{e}\leq F_{z}^{*}g\leq c g_{e}$ on $F_{z}^{-1}(B_{g}(z,R))$ where $g_{e}$ is the euclidean metric.
 \end{defn}
 \noindent By $(iii)$
 $$
 \forall\ w \in B_{g}(z,R)\ :\ {1\over c}\Vert F_{z}^{-1}(w)\Vert_{e}\leq d_{g}(w,z)\leq c\Vert F_{z}^{-1}(w)\Vert_{e}
 $$
\begin{rem} If an Hermitian manifold $(M,g)$ has bounded geometry then the geodesic exponential map $\exp_{z} : T^{\mathbb{R}}_{z}M\rightarrow M$ is defined on a ball $B(0,r)\subset T^{\mathbb{R}}_{z}M$ for any $r < R$ and provide a diffeomorphism of this ball onto the ball $B_{g}(z,r)\subset M$. It follows that the manifold $(M,g)$ is complete.
\end{rem}
\begin{rem}
It is well known that if $(M,g)$ has  bounded geometry and $Ric(g)\geq Kg$ then $(M,g)$ satisfy the
uniform ball size condition ([7] Prop. 14), i.e. for every $r\in\mathbb{R}^{+}$
$$
\inf_{z\in M}vol(B_{g}(z,r)) > 0\quad\hbox{and}\quad\sup_{z\in M}vol(B_{g}(z,r)) < \infty
$$
Also by volume comparison theorem [3], there are nonnegative constants $ C,\alpha,\beta$ such that
$$
vol_{g}(B_{g}(z,r))\leq Cr^{\alpha}e^{\beta r},\quad \forall\ r\geq 1,\ z\in M
$$
\end{rem}
\noindent Bounded geometry allows one to produce an exhausion function which behaves like the distance function and whose gradient and hessian are bounded on $M$ [23].
 \begin{lem}
 Let $(M,g)$ be a Hermitian manifold with bounded geometry. For every $z\in M$ there exists a smooth function $\Psi_{z} : M\longrightarrow\mathbb{R}$  such that \\
 (i) $C_{1}d_{g}(.,z)\leq \Psi_{z}\leq C_{2}(d_{g}(.,z)+1)$,\\
 (ii) $\vert\partial\Psi_{z}\vert_{g}\leq C_{3}$, and\\
 (iii) $-C_{4}\omega_{g}\leq i\partial\bar\partial\Psi_{z}\leq C_{5}\omega_{g}$.\\
 Furthermore, the constants in $(i), (ii)$ and $(iii)$ depend only on the constants associated with the bounded geometry of $(M,g)$.
 \end{lem}
 \subsection{Bergman Kernel of Line Bundles}
Let $L$ be a holomorphic  hermitian  line bundle over a complex manifold $M$, and let $(U_{j})$ be a covering of the manifold by open sets over which $L$ is locally trivial. A section $s$ of $L$ is
then represented by a collection of complex valued functions $s_{j}$ on $U_{j}$ that are
related by the transition functions $(g_{jk})$ of the bundle
$$
s_{j}=g_{jk}s_{k}\quad\hbox{on $U_{j}\cap U_{k}$}
$$
 We say that $s$ is holomorphic if each $s_{i}$ is holomorphic on $U_{j}$ and we write $\bar\partial s=0$. The conjugate bundle of $L$  is the  hermitian anti-holomorphic line bundle $\overline{L}$ whose transition functions are $(\overline{g}_{jk})$. A metric  $h$ on $L$ is
given by a collection of real valued functions  $\Phi_{j}$ on $U_{j}$ , related so that
$$
\vert f_{j}\vert^{2}e^{-\Phi_{j}} =:\vert s\vert^{2}_{h}
$$
is globally well defined. We will write  $h$ for the collection $(\Phi_{j})$ , and refer to $h$ the metric  on $L$. We say that L is positive,
$L>0$, if $h$ can be chosen smooth with curvature $$c(L):=i\partial\bar\partial\Phi_{j}$$ strictly positive, and that
$L$ is semipositive, $L\geq 0$, if it has a smooth metric of semipositive curvature. We say that $h$ is a singular metric if each $\Phi_{j}$ is only plurisubharmonic.
\begin{defn} A holomophic Hermitian line bundle $(L,h)\longrightarrow (M,g)$ has bounded curvature if
$$
-M\omega_{g}\leq c(L)\leq M\omega_{g}
$$
for some positive constant $M$.
\end{defn}
\noindent Fix $p\in [1,+\infty]$.  Let the Lebesgue spaces
\begin{eqnarray*}
L^{p}(M, L) &:=& \{ s : M\longrightarrow L\ :\  \Vert s\Vert_{p}:=\Bigl(\int_{M}\vert u\vert^{p}_{h}dv_{g}\Bigr)^{1\over p} < \infty \}\\
L^{\infty}(M,L)&:=& \{ s : M\longrightarrow   L\ :\  \Vert s\Vert_{\infty}:=\sup_{z\in M}\vert s(z)\vert_{h}  < \infty\},
\end{eqnarray*}
and the Bergman spaces of holomorphic sections
\begin{eqnarray*}
\mathcal{F}^{p}(M, L)&:=&\{s\in L^{p}(M,L)\ :\ \bar\partial s=0\}\\   \mathcal{F}^{\infty}(M,L)&:=&\{s\in L^{\infty}(M,L)\ :\ \bar\partial s=0\}
\end{eqnarray*}
Let us note an important property of the space $\mathcal{F}^{2}(X,L)$ which follows from the Cauchy estimates for holomorphic functions. Namely, for every compact set $G\subset M$ there exists $C_{G} > 0$ such that
 \begin{equation}
 \sup_{z\in G}\vert s(x)\vert\leq C_{G}\Vert s\Vert_{2}\ \hbox{for all $s\in \mathcal{F}^{2}(X,L)$}
 \end{equation}
 We deduce that $\mathcal{F}^{2}(M,L)$ is a closed subspace of $L^{2}(M,L)$ one can show also that $\mathcal{F}^{2}(M,L)$ is separable ( cf. [31, p.30]).
 \begin{defn}
 \noindent The Bergman projection is the orthogonal projection
 $$
 P\ :\ L^{2}(M,L)\longrightarrow\mathcal{F}^{2}(M,L)
 $$
 \end{defn}
\noindent  In view of (2.1) , the Riesz representation theorem shows that for a fixed $z\in M$ there exists a section $K(z,.)\in L^{2}(M,  L_{z}\otimes\overline{L})$ such that
 \begin{equation}
 s(z)=\int_{M}K(z,w)\bullet s(w)dv_{g}\ \hbox{for all}\ s\in \mathcal{F}^{2}(M, L)
 \end{equation}
 \noindent The distribution kernel $K$ is called the Bergman Kernel of $( L,h)\longrightarrow (M,g)$. If $\mathcal{F}^{2}(M, L)=0$ we have of course $K(z,z)=0$ for all $z\in M$. If $\mathcal{F}^{2}(M,L)\not=0$, consider an orhonormal basis $(s_{j})_{j=1}^{d}$ of $\mathcal{F}^{2}(X,L)$ ( where $1\leq d\leq\infty$). By estimates (1.1)
\begin{eqnarray*}
K(z,w)=\sum_{j=1}^{d}s_{j}(z)\otimes\overline{s_{j}(w)}\in L_{z}\otimes \overline{L_{w}}
\end{eqnarray*}
where the right hand side converges on every compact together with all its derivates ( ses [31, p.62]). Thus $K(z,w)\in C^{\infty}(M\times M,L\otimes\overline{L})$. It follows that
$$
(Ps)(z)=\int_{M} K(z,w)\bullet s(w) dv_{g}(w),\ \hbox{for all}\ s\in L^{2}(M,L),
$$
that is $K(.,.)$ is the intergral kernel of the Bergman projection $P$. Since
\begin{eqnarray*}
 \vert K(z,w)\vert^{2}&=&\sum_{j=1}^{d}\sum_{k=1}^{d}<s_{j}(z)\otimes\overline{s_{j}(w)},s_{k}(z)\otimes\overline{s_{k}(w)}>_{L_{z}\otimes \overline{L_{w}}}\\
&=& \sum_{j}\sum_{k}<s_{j}(z),s_{k}(z)>_{L_{z}}\overline{<s_{j}(w),s_{k}(w)>_{L_{w}}}
\end{eqnarray*}
then $K(z,w)$ is symetric
$$
\vert K(z,w)\vert=\vert K(w,z)\vert
$$
The function $\vert K(z,z)\vert$ is called the Bergman function of $\mathcal{F}^{2}(M,L)$. It satisfies
$$
 \vert K(z,z)\vert=\int_{M}\vert K(z,w)\vert^{2}dv_{g}(w)
$$
\subsection{$\bar\partial$-Methods}
\noindent We  recall Demailly's theorem [9], which generalizes H\"ormander's $L^{2}$ estimates [13]
(Theorem 2.2.1, p. 104) for forms with values in a line bundle.
\begin{thm} Let $(X,\omega)$ be a complete K\"ahler manifold, $(L,h)$ a holomorphic hermitian line bundle over $X$, and let $\phi$ be a locally
integrable function over $X$. If the curvature $c(L)$ is such that
$$
c(L)+Ric(\omega)+i\partial\bar\partial\phi \geq\gamma\omega
$$
for some positive and continuous function $\gamma$ on $X$, then for all $v\in L^{2}_{(0,1)}(X,L,loc)$, $\bar\partial$-closed and such that
$$
\int_{X}\gamma^{-1}\vert v\vert^{2}dv_{\omega} < \infty
$$
there exists  $u\in L^{2}(X,L)$ such that
$$
\bar\partial u=v\quad\hbox{and}\quad\int_{X}\vert u\vert_{h}^{2}dv_{\omega}\leq\int_{X}\gamma^{-1}\vert v\vert_{\omega,h}^{2}dv_{\omega}
$$
\end{thm}
\noindent Also, we  recall J.McNeal-D.Varolin's theorem [19](Theorem 2.2.1, p. 104), which generalizes Berndtsson-Delin's improved $L^{2}$-estimate  of $\bar\partial$-equation having minimal $L^{2}$-norm [2],[8] for forms with values in a line bundle.
\begin{thm}
Let $(M,g)$ be a Stein K\"ahler manifold, and $(L,h)\longrightarrow(M,g)$ a holomorphic hermitian line
bundle with Hermitian metric $h$. Suppose there exists a smooth function $\eta : M \rightarrow \mathbb{R}$
and a positive, a.e. strictly positive Hermitian $(1,1)$-form $\Theta$ on $M$ such that
$$
c(L)+Ric(g)+i\partial\bar\partial\eta-i\partial\eta\wedge\bar\partial\eta\geq\Theta
$$
Let $v$ be an $L$-valued $(0,1)$-form such that $v=\bar\partial u$ for some $L$-valued section $u$ satisfying
$$
\int_{M}\vert u\vert_{h}^{2}dv_{g} < \infty
$$
Then the solution $u_{0}$ of $\bar\partial u=v$ having minimal $L^{2}$-norm i.e
$$
\int_{M}<u_{0},\sigma>dv_{g}=0\ \hbox{for all}\ \sigma\in \mathcal{F}^{2}(M,L)
$$
satisfies the estimate
$$
\int_{M}\vert u_{0}\vert_{h}^{2}e^{\eta}dv_{g}\leq\int_{M}\vert v\vert^{2}_{\Theta,h}e^{\eta}dv_{g}
.$$
\end{thm}
\section{Estimates for the Bergman Kernel}
 \subsection{Weighted Bergman Inequalities}
 \begin{prop} Let $(M,g)$ be a  complete noncompact K\"ahler  manifold with bounded geometry and lower Ricci curvature bound. Let $(L,h)\longrightarrow (M,g)$ be a hermitian holomorphic line bundle with bounded curvature. Fix $p\in ]0,\infty[$.
Then for each $r > 0$ there exists a constant $C_{r}$ such that if $s\in \mathcal{F}^{2}(M,L)$ then
\begin{equation}
\vert s(z)\vert^{p}\leq C_{r}^{p}\int_{B_{g}(z,r)}\vert s\vert^{p}dv_{g}
\end{equation}
in particular $\mathcal{F}^{p}(M,L)\subset\mathcal{F}^{\infty}(M,L)$
and
\begin{equation}
\vert\nabla\vert s(z)\vert^{p}\vert_{g}(z)\leq C_{r}^{p}\int_{B_{g}(z,r)}\vert s\vert^{p}dv_{g}
\end{equation}
\end{prop}
\begin{proof}
 Since $(M,g)$ has bounded geometry there exists positive numbers $R$ and  $c$  such that for all $z\in M$ there exists a biholomorphic mapping $\Psi_{z} : (U,0)\subset\mathbb{C}^{n}\longrightarrow (V,z)\subset M$ such that \\
 (i) $\Psi_{z}(0)=z$,\\
 (ii) $B_{g}(z,R)\subset \Psi_{z}(U)$ and \\
 (iii) ${1\over c} g_{e}\leq \Psi_{z}^{*}g\leq c g_{e}$ on $\Psi_{z}^{-1}(B_{g}(z,R))$ where $g_{e}$ is the euclidean metric.\\
 Consider the $(1,1)$-forme  defined on $B_{e}(0,\delta(R))\subset\subset\Psi_{z}^{-1}(B_{g}(z,R))\subset\mathbb{C}^{n}$ by
 $$
 \Theta:=\Psi_{z}^{*}c(L)
 $$
 Since $-K\omega_{g}\leq c(L)\leq K\omega_{g}$, by [25] Lemma 4.1 there exists a function $\phi\in C^{2}(B_{e}(0,\delta))$ such that
 $$
 i\partial\bar\partial\phi=\Theta\quad\hbox{and}\quad\sup_{B_{e}(0,\delta)}(\vert\phi\vert+\vert d\phi\vert_{g_{e}})\leq M
 $$
 On $B_{g}(z,\eta)\subset\subset\Psi_{z}(B_{e}(0,\delta(R))$, consider the   $C^{2}$-function $$\psi:=\phi\circ\Psi_{z}^{-1}$$ By (iii) we have
 $$
 i\partial\bar\partial\psi=c(L)\quad\hbox{and}\quad\sup_{B_{g}(z,\eta)}(\vert\psi\vert+\vert \nabla\psi\vert_{g})\leq M^{'}
 $$
 where $M^{'}$ and $\eta$ depend only on $R$ and $c$.\\
 Let $e$ be a  frame of $L$ arround $z\in B_{g}(z,\eta)$  and $\Phi(w)=-\log\vert e(w)\vert^{2}$. Then $i\partial\bar\partial\psi=i\partial\bar\partial\Phi$ on $B_{g}(z,\eta)$. Hence the function
 $$
\rho(w)=\Phi(w)-\Phi(z)+\psi(z)-\psi(w)
 $$
 is pluriharmonic. Then $\rho=\Re(F)$ for some holomorphic function $F$ with $\Im(F)(z)=0$ and
 \begin{equation}
 \sup_{B_{g}(z,\eta)}\vert\Phi-\Phi(z)-\Re(F)\vert=\sup_{B_{g}(z,\eta)}\vert\psi-\psi(z)\vert\leq C
 \end{equation}
 \begin{equation}
 \sup_{B_{g}(z,\eta)}\vert\nabla(\Phi-\Phi(z)-\Re(F))\vert_{g}=\sup_{B_{g}(z,\eta)}\vert\nabla\psi\vert_{g}\leq C
 \end{equation}
We can suppose $0 < r \leq \eta$. According to [17] , for all $z\in M$ and all holomorphic function $f$ on $B_{g}(z,\eta)$ and all $\zeta\in B_{g}(z,\eta/2)$
$$
\vert f(\zeta)\vert^{p} \leq {C\over\hbox{Vol}(B_{g}(\zeta,\eta/2))}\int_{B_{g}(\zeta,\eta)}\vert f(w)\vert^{p} dv_{g}
$$
where $C$ depend only in $K,n,\eta$. Since $g$ has sbounded geometry $\hbox{Vol}(B_{g}(z,\eta/2))\succeq 1$ uniformly in $z$. Hence
$$
\vert f(\zeta)\vert^{p}\leq  C\int_{B_{g}(\zeta,\eta)}\vert f(w)\vert^{p} dv_{g}
$$
Let $s\in \mathcal{F}^{p}(M,L)$ and $s=fe$ on $B_{g}(z,\eta)$. By (2.3) we have have
\begin{eqnarray*}\vert s\vert^{p}_{h}&=&\vert fe^{-{F\over 2}}\vert^{p} e^{-{p\over 2}\Phi(z)}e^{-{p\over 2}(\Phi-\Phi(z)-\Re(F))}\\
&\leq & C^{p} \vert fe^{-{F\over 2}}\vert^{p} e^{-{p\over 2}\Phi(z)}
\end{eqnarray*}
By mean value inequality
\begin{eqnarray*}
\vert f(z)e^{-{F(z)\over 2}}\vert^{p}e^{-{p\over 2}\Phi(z)}&\leq & c_{r}^{p}\int_{B_{g}(z,r)}\vert fe^{-{F\over 2}}\vert^{p}e^{-{p\over 2}\Phi(z)}dv_{g}\\
&\leq & C_{r}^{p}\int_{B_{g}(z,r)}\vert fe^{-\Phi(w)}\vert^{p} dv_{g}
\end{eqnarray*}
Hence
$$
\vert s(z)\vert^{p}_{h}\leq C_{r}^{p}\int_{B_{g}(z,r)}\vert s\vert^{p}dv_{g}
$$
By (2.3) and (2.4)
\begin{eqnarray*}
\vert\nabla\vert s\vert_{h}^{p}\vert_{g} &\leq & e^{-{p\over 2}\Phi(z)}e^{-{p\over 2}(\Phi-\Phi(z)-\Re(F)))}\vert\nabla\vert fe^{-{F\over 2}}\vert^{p}\vert\\ &+&{p\over 2}\vert fe^{-{F\over 2}}\vert^{p} e^{-{p\over 2}\Phi(z)}e^{-{p\over 2}(\Phi-\Phi(z)-\Re(F))}\vert\nabla(\Phi-\Phi(z)-\Re(F))\vert_{g}\\
&\leq & e^{-{p\over 2}\Phi(z)}e^{-{p\over 2}(\Phi-\Phi(z)-\Re(F))}\vert\nabla\vert fe^{-{F\over 2}}\vert^{p}\vert\\ &+&{p\over 2}\vert s\vert_{h}^{p}e^{-{p\over 2}(\Phi-\Phi(z)-\Re(F))}\vert\nabla(\Phi-\Phi(z)-\Re(F))\vert_{g}\\
&\leq & C^{p}\bigl(e^{-{p\over 2}\Phi(z)}\vert\nabla\vert fe^{-{F\over 2}}\vert^{p}\vert + {p\over 2}\vert s\vert_{h}^{p}\bigr)
\end{eqnarray*}
By mean value inequality ( Cauchy formula for partial derivates ), there exists $c_{r} > 0$ such that
\begin{eqnarray*}
\vert\nabla\vert fe^{-{F\over 2}}\vert^{p}\vert(z)e^{-{p\over 2}\Phi(z)} &\leq & c^{p}_{r}\int_{B_{g}(z,r)}\vert fe^{-{F\over 2}}\vert^{p}e^{-{p\over 2}\Phi(z)}dv_{g}\\
&\leq & C_{r}^{p}\int_{B_{g}(z,r)}\vert s\vert^{p}dv_{g}
\end{eqnarray*}
From this , we get (2.2).
\end{proof}
\subsection{Slow Growth of Bergman Sections}
\noindent
\begin{lem}. Let $(M,g)$ be a  K\"ahler manifold with  bounded geometry and lower Ricci curvature Bound. Let $(L,h)\longrightarrow (M,g)$ be a hermitian holomorphic line bundle with bounded curvature.
Then there exists $\delta > 0$ with the following properties :  if  $z\in M,\ s\in\mathcal{F}^{p}(M,L),\ \Vert s\Vert_{p}\leq 1$ then
$$\vert s(z)\vert_{h}\geq a\Longrightarrow\vert s(w)\vert_{h}\geq{a\over 2},\ \forall\ w\in B_{g}(z,\delta).
$$
\end{lem}
\begin{proof}
Le $ R >\delta > 0$. By (3.2) of proposition 3.1  and mean value theorem for all $w\in B_{g}(z,R/2)$
\begin{eqnarray*}
\vert\vert s(w)\vert^{p}_{h}-\vert s(z)\vert^{p}_{h}\vert &\leq & C^{p}_{r}d_{g}(w,z)\Bigl(\int_{B_{g}(z,R)}\vert s(\zeta)\vert^{p}dv_{g}\Bigr)\\
&\leq & \delta C_{R}^{p}\Vert s\Vert_{p}^{p}
\end{eqnarray*}
Hence if $\delta$ is small enough
$$
\forall w\in B_{g}(z,\delta)\ :\ \vert s(w)\vert_{h}^{p}\geq a^{p}-\delta C_{R}^{p}\geq {a^{p}\over 2^{p}}
$$
\end{proof}
\subsection{One-Point Interpolation with Uniform $L^{p}$ Estimates}
\begin{prop} Let $(M,g)$ be a K\"ahler Cartan-Hadamard manifold with bounded geometry.
Let $(L,h)\longrightarrow (M,g)$ be a hermitian holomorphic line bundle with bounded curvature such that
$$
c(L)+Ricci(g)\geq a\omega_{g}
$$
for some positive constant $a$. Let $p\in [1,+\infty]$. If $p\not=2$ or $p\not=\infty$,  suppose further  that
$$
\sup_{z\in M}\int_{M}e^{-\beta d_{g}(w,z)}dv_{g} < \infty
$$
for all $\beta > 0$. Then there exists $ C > 0$ such that for each $z\in M$ and  $\lambda\in L_{z}$ there exists $s\in\mathcal{F}^{p}(M,L)$ such that $$ s(z)=\lambda\ \hbox{and}\  \Vert s\Vert_{p}\leq C\vert\lambda\vert_{L_{z}} $$
\end{prop}
\begin{proof}
Let $z\in M$ and fix a smmoth function $\chi$ with compact support on $B_ {g}(z,Rc^{-1}/2)$ ( $R$ and $c$ are  constants in definition 2.3)  such that \\
(i) $0\leq\chi\leq 1$,\\
(ii) $\chi\mid_{B_{g}(z,Rc^{-1}/4)}=1$,\\
(iii) $\vert\bar\partial\chi\vert_{g} \preceq 1$  \\
Let $s_{0}$ be a  holomorphic section of $L$ arround $B_{g}(z,\eta)$ such thet $s_{0}(z)\not=0$.
Since $(L,h)$ is $g$-regular, for all $w\in B_ {g}(z,\eta)$
$$
\Phi(w)\simeq \Phi(z)+\Re(F(w))
$$
Let $\lambda\in L_{z}=\{z\}\times\mathbb{C}$. Consider the local section
$$
s_{z}(w)=\lambda(w)e^{\Phi(z)+\Re(F(w)}s_{0}(w)
$$
and the $(0,1)$-forme with values on $L$
$$
v(w)=\bar\partial(\chi.s_{z})(w)=\bar\partial\gamma(w).s_{z}(w)
$$
Let $\Phi_{z}\in C^{\infty}(M)$ as in lemma 2.8 and choose $\epsilon > 0$ small enough such that
$$
c(L)+Ricci(g)-\epsilon\partial\bar\partial\Phi_{z}\geq g\ \hbox{on $M$}
$$
By (iii) in Definition 2.5 of bounded geometry
 $$c^{-2n}dv_{e}\leq \Psi_{z}^{*}dv_{g}\leq c^{2n}dv_{e}\ \hbox{on}\  \Psi_{z}^{-1}(B_{g}(z,\eta))$$ Hence
$$
 \hbox{Vol}_{g}(B(z, \eta ))\asymp 1\ \hbox{uniformly in $z\in M$}
$$
Since $M$ is Cartan-Hadamard $d^{2}_{g}(.,z)$ is smooth. By comparison theorem for the Hessian [10] the function $ w\in M\rightarrow\phi_{z}(w):= \log d^{2}_{g}(z,z)$  is plurisousharmonic on $M$.
\begin{eqnarray*}
\int_{M}\vert v\vert^{2} e^{\epsilon\Phi_{z}}e^{-2n\phi_{z}}dv_{g}&\preceq &\int_{B_{g}(z,\eta/2)\setminus B_{g}(z,\eta/4)}{\vert v\vert^{2} e^{2\epsilon\Phi_{z}}\over d_{g}(.,z)^{2n}}dv_{g}\\
&\preceq &\vert\lambda\vert^{2}\hbox{Vol}_{g}(B_{g}(z,\eta/2)\setminus B_{g}(z,\eta/4))\\
&\preceq &\vert\lambda\vert^{2}e^{-\Phi(z)}=\vert\lambda\vert^{2}_{L_{z}}
\end{eqnarray*}
 Since a K\"ahler Cartan-Hadamard manifold is Stein [32], by Theorem 2.8 there exists $u$ such that $\bar\partial u=v$ and
$$
\int_{M}{\vert u(w)\vert^{2}e^{2\epsilon\Phi_{z}(w)}\over d_{g}(w,z)^{2n}}dv_{g}\preceq\vert \lambda\vert^{2}_{L_{z}}
$$
Since $w\longrightarrow d_{g}^{-2n}(w,z)$ is not summable near $z$, we have $u(z)=0$. Let
$$
s(w)=\chi(w)s_{z}(w)-u(w)
$$
Then $s(z)=\lambda$ and $\bar\partial s=0$. Since $(2n)!e^{t}\geq t^{2n}$ if $t\geq 0$ and $\Phi_{z}\asymp d_{g}(.,z)$
$$
\int_{M}\vert u\vert^{2}dv_{g}\preceq \int_{M}{\vert v\vert^{2} e^{2\epsilon\Phi_{z}}\over d_{g}(.,z)^{2n}}dv_{g}\preceq\vert\lambda\vert^{2}_{L_{z}}
$$
Thus
$$
\int_{M}\vert s\vert^{2} dv_{g}\leq 2\int_{M}\vert \chi s_{z}\vert^{2}dv_{g}+2\int_{M}\vert u\vert^{2}dv_{g}\leq C\vert\lambda\vert^{2}_{L_{z}}
$$
Also
\begin{eqnarray*}
\int_{M}\vert s(w)\vert^{2}e^{\epsilon\Psi_{z}(w)}dv_{g}(w)&\leq &\int_{M}\vert\chi(w)\vert^{2}e^{\epsilon\Phi_{z}(w)}\vert e(w)\vert^{2} dv_{g}\\
&{}& + \int_{M}{\vert u(w)\vert^{2}e^{\epsilon\Phi_{z}(w)}\over d_{g}(w,z)^{2n}}d_{g}^{2n}(w,z)e^{-\epsilon\Phi_{z}(w)}dv_{g}(w)
\end{eqnarray*}
Since  $\vert\Phi(w)\simeq \Phi(z)+\Re(F(w))$   and  $\Psi_{z}(w)\asymp d_{g}(w,z)\asymp 1$ uniformly on the support of $\gamma$ and $d_{g}^{2n}(w,z)e^{-\epsilon\Phi_{z}(w)}\asymp 1$ uniformly in $z\in M$, there exists $ C > 0$  independent of $z$ such that
$$
\int_{M}\vert\chi(w)\vert^{2}e^{\epsilon\Phi_{z}(w)}\vert s_{z}(w)\vert^{2} dv_{g}\leq C \vert \lambda\vert^{2}_{L_{z]}}
$$
and
$$
\int_{M}{\vert u(w)\vert^{2}e^{\epsilon\Phi_{z}(w)}\over d_{g}(w,z)^{2n}}d_{g}^{2n}(w,z)e^{-\epsilon\Phi_{z}(w)}dv_{g}(w)\leq C\vert \lambda\vert^{2}_{L_{z]}}
$$
Hence
$$
\int_{M}\vert s(w)\vert^{2}e^{\epsilon\Phi_{z}(w)}dv_{g}(w)\leq C\vert\lambda\vert^{2}_{L_{z}}
$$
Since $\Vert\partial\bar\partial\Phi_{z}\Vert_{\infty}$ is  uniformly bounded in $z\in M$,  the line bundle $(L,he^{\epsilon\Phi_{z})})$ has bounded curvature. By $(3.1)$ of proposition 3.1
\begin{eqnarray*}
\vert s(w)\vert^{2}&\preceq &\vert s(w)\vert^{2}e^{\epsilon\Phi_{z}}\\
&\preceq &\int_{B_{g}(w,\eta)}\vert s(\zeta)\vert^{2}e^{\epsilon\Phi_{z}(\zeta)}dv_{g}(\zeta)\\
&\preceq & \int_{M}\vert s(\zeta)\vert^{2}e^{\epsilon\Phi_{z}(\zeta)}dv_{g}(\zeta)\\
&\leq & C\vert\lambda\vert_{L_{z}}^{2}
\end{eqnarray*}
Hence $\Vert s\Vert_{\infty}\leq C\vert\lambda\vert_{L_{z}}$. Also
\begin{eqnarray*}
\vert s(w)\vert^{2}e^{\epsilon\Phi_{z}(w)} &\leq & C_{R} \int_{B_{g}(w,R)}\vert s(\zeta)\vert^{2}e^{\epsilon\Phi_{z}(\zeta)}dv_{g}(\zeta) \\
&\leq & C_{R} \int_{M}\vert s(\zeta)\vert^{2}e^{\epsilon\Phi_{z}(\zeta)}dv_{g}(\zeta)\\
&\leq & C_{R} \vert\lambda\vert_{L_{z}}^{2}
\end{eqnarray*}
Thus
 \begin{eqnarray*}
\int_{M}\vert s\vert^{p}dv_{g}&=&\int_{M}\Bigl(\vert s\vert^{2}e^{\epsilon\Phi_{z}}\Bigr)^{p\over 2}e^{-{p\over 2}\epsilon\Phi_{z}}dv_{g}\\
&\leq & C\vert\lambda\vert^{p}\int_{M}e^{-{p\over 2}\epsilon\Phi_{z}}dv_{g}\\
&\leq & C\vert\lambda\vert^{p}_{L_{z}}\int_{M}e^{-{pC_{1}\over 2}\epsilon d_{g}(w.z)}dv_{g}(w)\\
&\leq & C^{p}\vert\lambda\vert_{L_{z}}^{p}
\end{eqnarray*}
Finally, there exists  $C > 0$ independent of $z\in M$ and $p\in [1,+\infty]$ such that $$\Vert s\Vert_{p}\leq C\vert\lambda\vert_{L_{z}}.$$
\end{proof}
\subsection{Diagonal Bounds for the Bergman Kernel}
As a consequence of $(3.1)$ proposition 3.1 and proposition 3.3, we obtain the following proposition.
\begin{prop} Let $(M,g)$ be a  K\"ahler manifold with  bounded geometry and lower Ricci curvature Bound. Let $(L,h)\longrightarrow (M,g)$ be a hermitian holomorphic line bundle with bounded curvature. There is a constant $C > 0$ such that for all $z\in M\ :\  \vert K(z,z)\vert\preceq C$.
Therefore $\vert K(z,w)\vert\leq  C$ for all $z,w \in M$.
\end{prop}
\begin{proof}
Let $(s_{j}]$ be a orhonormal basis of $\mathcal{F}^{2}(M,L)$. By definition of the Bergman Kernel
$$
K(z,w)=\sum_{j}s_{j}(z)\otimes\overline{s_{j}(w)}
$$
By $(3.1)$ proposition 3.1 the evaluation
\begin{eqnarray*}
ev_{z} &:& \mathcal{F}^{2}(M,L)\longrightarrow L_{z}\\
&{}& s\longrightarrow s(z)
\end{eqnarray*}
is continuous and
$$\vert K(z,z)\vert \preceq 1$$ uniformly in $z\in M$.  Therefore
\begin{eqnarray*} \vert K(z,w)\vert&\leq &\sum_{j}\vert s_{j}(z)\vert\vert s_{j}(w)\vert\\
&\leq &\sqrt{\vert K(z,z)\vert}\sqrt{\vert K(w,w)\vert}\preceq 1
\end{eqnarray*}
 \end{proof}
\noindent The following result gives bounds for the Bergman kernel in a small but uniform neighborhood of the diagonal
\begin{prop} Let $(M,g)$ be a  K\"ahler manifold with  bounded geometry and lower Ricci curvature Bound. Let $(L,h)\longrightarrow (M,g)$ be a hermitian holomorphic line bundle with bounded curvature. There are constants $\delta, C_{1} , C_{2} > 0$ such that for all $z \in M$  and $w\in B_{g}(z,\delta)$
$$
C_{1}\vert K(z,z)\vert\leq\vert K(z,w)\vert\leq C_{2}\vert K(z,z)\vert
$$
\end{prop}
\begin{proof}
Let $z\in M$. Fix a frame $e$ in a neighborhood $U$ of the point $z$ and consider an orhonormal basis $(s_{j})_{j=1}^{d}$ of $\mathcal{F}^{2}(X,L)$ ( where $1\leq d\leq\infty$). In $U$ each $s_{i}$ is represented by a holomorphic function $f_{i}$ such that $s_{i}(x)=f_{i}(x)e(x)$. Let
$$
s_{z}(w):=\vert e(z)\vert\sum_{i=1}^{d}\overline{f_{i}(z)}s_{i}(w)
$$
Then
\begin{eqnarray*}
\vert s_{z}(w)\vert&=&\Big\vert\Bigl(\sum_{i=1}^{d}\overline{f_{i}(z)}s_{i}(w)\Bigr)\otimes\overline{e(z)}\Big\vert\\
&=&\Big\vert\sum_{i=1}^{d}s_{i}(w)\otimes\overline{s_{i}(z)}\Big\vert\\
&=&\vert K(w,z)\vert
\end{eqnarray*}
and
\begin{eqnarray*}
\int_{M}\vert s_{z}\vert^{2}dv_{g}(w)&=&\int_{M}\vert K(w,z)\vert^{2}dv_{g}(w)\\
&=&\vert K(z,z)\vert\preceq 1
\end{eqnarray*}
Hence, by lemma 3.2, there exists $C,\delta > 0$ independant of $z$ such that
$$
\vert K(w,z)\vert=\vert s_{z}(w)\vert\geq C\vert s_{z}(z)\vert=C\vert K(z,z)\vert
$$
for all $w\in B_{g}(z,\delta)$.
\end{proof}
\subsection{Off-Diagonal Decay of the Bergman Kernel}
A key tool we use is the following  off-diagonal upper bound exponential decay  for the Bergman kernel of $L$.
\begin{thm} Let $(M,g)$ be a Stein K\"ahler manifold with bounded geometry.
Let $(L,h)\longrightarrow (M,g)$ be a hermitian holomorphic line bundle with bounded curvature such that
$$
c(L)+Ricci(g)\geq a\omega_{g}
$$
for some positive constant $a$. There are constants $\alpha,\ C > 0$ such that for all $z,w \in M$ ,
$$
\vert K(z,w)\vert\leq Ce^{-\alpha d_{g}(z,w)}
$$
\end{thm}
\begin{proof}
Let $z,w\in M$ such that $d_{g}(z,w) \geq \delta$ where $\delta > 0$ as in Proposition 3.4. Fix a smooth function  $\chi\in C^{\infty}_{0}(B_{g}(w,\delta/2))$ such that  \\ (i)$\ 0\leq\chi\le 1$,\\
(ii)$\ \chi=1$ in $B_{g}(w,\delta/4)$,\\
(iii)$\ \vert\bar\partial\chi\vert_{g}\preceq\chi$.\\
 Let $s_{z}\in \mathcal{F}^{2}(M,L)$ defined by
$$
s_{z}(w):=\vert e(z)\vert\sum_{i=1}^{d}\overline{f_{i}(z)}s_{i}(w)
$$
where $(s_{i})_{1\leq i \leq d}$ is an orthonormal basis of $\mathcal{F}^{2}(M,L)$ and $e$ is a a local vframe of $L$ arround $z$. Then  $\vert s_{z}(w)\vert=\vert K(w,z)\vert$ and $\Vert s_{z}\Vert_{2}=\vert K(z,z)\vert\preceq 1$. Also
$$
s_{z}(w)\otimes\overline{e(z)\over\vert e(z)\vert}=K(w,z)
$$
By (3.1) Proposition 3.1
$$
\vert s_{z}(w)\vert^{2} \preceq \int_{B(w,\delta/2)}\chi(\zeta)\vert s_{z}(\zeta)\vert^{2}dv_{g}
\preceq \Vert s_{z}\Vert^{2}_{L^{2}(\chi dv_{g})}
$$
We have  $\Vert s_{z}\Vert_{L^{2}(\chi dV_{g})}=\sup_{\sigma}\vert<\sigma,s_{z}>_{L^{2}(\chi dv_{g})}\vert$ where $\sigma\in\mathcal{F}^{2}(B_{g}(z,\delta),L)$ such that $\Vert \sigma\Vert_{L^{2}(\chi dv_{g})}=1$.
we have
\begin{eqnarray*}
\Big\vert<\sigma,s_{z}>_{L^{2}(\chi dv_{g})}\Big\vert_{\mathbb{C}} &=& \Big\vert\int_{M}<\chi(w)\sigma(w),s_{z}(w)>dv_{g}(w)\Big\vert_{\mathbb{C}}\\
&=& \Big\vert\sum_{i=1}^{d}\int_{M}<\chi(w)\sigma(w),\vert e(z)\vert\overline{f_{i}(z)}s_{i}(w)>dv_{g}(w)\Big\vert_{\mathbb{C}}\\
&=&\Big\vert\sum_{i=1}^{d}\int_{M}<\chi(w)\sigma(w), s_{i}(w)>f_{i}(z)\vert e(z)\vert dv_{g}(w)\Big\vert_{\mathbb{C}}\\
&=&\Big\vert\sum_{i=1}^{d}\int_{M}<\chi(w)\sigma(w), s_{i}(w)>f_{i}(z)e(z)dv_{g}(w)\Big\vert_{L_{z}}\\
&=&\Big\vert\sum_{i=1}^{d}\int_{M}<\chi(w)\sigma(w), s_{i}(w)>s_{i}(z)dv_{g}(w)\Big\vert_{L_{z}}\\
&=&\Big\vert\int_{M} K(z,w)\bullet\chi(w)\sigma(w)dv_{g}(w)\Big\vert_{L_{z}}\\
&=&\vert P(\chi\sigma)(z)\vert_{L_{z}}
\end{eqnarray*}
Since $c(L)+Ricci(g)\ge ag$,
by Theorem 2.11 there exists a solution $u$ of $\bar\partial u=\bar\partial\chi.\sigma$ such that
$$
\int_{M}\vert u\vert^{2}dv_{g}
\preceq\int_{M}\vert\bar\partial\chi\vert_{g}^{2}\vert \sigma\vert^{2}dv_{g} < \infty
$$
Let $u_{\sigma}=\chi\sigma-P(\chi\sigma)$ be  the  solution having  minimal  $L^{2}$-norm of $$\bar\partial u=\bar\partial\chi.\sigma$$ Since $\chi(z)=0$
$$
\Big\vert<\sigma,s_{z}>_{L^{2}(\chi dv_{g})}\Big\vert_{\mathbb{C}}=\vert P(\chi\sigma)(z)\vert_{L_{z}}=\vert u_{\sigma}(z)\vert_{L_{z}}
$$
Since $B(z,\delta/2)\cap B(w,\delta/2)=\emptyset$, the section $u_{\sigma}$ is holomorphic in $B_{g}(z,\delta/2)$. Let $\epsilon \in ]0,2/\delta]$, By (3.1) Proposition 3.1
\begin{equation}
\vert u_{\sigma}(z)\vert^{2}_{L_{z}}\preceq \int_{B_{g}(z,\delta/2)}\vert u_{\sigma}(\zeta)\vert^{2}_{L_{\zeta}}dv_{g}\preceq\int_{B_{g}(z,\delta/2)}e^{-\epsilon d(\zeta,z)}\vert u_{\sigma}(\zeta)\vert^{2}_{L_{\zeta}}dv_{g}
\end{equation}
\noindent Let $\eta:=-\epsilon\Phi_{z}$ where $\Phi_{z}$ is as in lemma 2.8 and $\Theta=\epsilon\omega_{g}$. Choose $\epsilon$ small enough such that
$$
c(L)+Ricci(g)-i\epsilon\partial\bar\partial\Phi_{z}-i\epsilon^{2}\partial\Phi_{z}\wedge\bar\partial\Phi_{z}-\epsilon\omega_{g}\geq 0
$$
By Theorem 2.12
$$
\int_{M}e^{- \epsilon\Phi_{z}}\vert u_{\sigma}\vert^{2}dv_{g}\preceq\int_{M}e^{-\epsilon\Phi_{z}}\vert \bar\partial \chi\vert_{g}^{2}\vert \sigma\vert^{2}dv_{g}
$$
Since  $C_{1}d_{g}(.,z)\leq \Phi_{z} \leq C_{2}(d_{g}(.,z)+1)$, we obtain
$$
\vert u_{\sigma}(z)\vert^{2}_{L_{z}}\preceq\int_{M}e^{-\epsilon C_{1}d_{g}(\zeta,z)} \chi(\zeta)\vert \sigma(\zeta)\vert^{2}dv_{g}
$$
Since $\zeta\in B_{g}(w,\delta) $ we have
\begin{eqnarray*}
d_{g}(\zeta,z)&\geq & d_{g}(z,w)-d_{g}(w,\zeta)\\
&\succeq & d_{g}(z,w)-\delta\succeq  d_{g}(z,w)
\end{eqnarray*}
Finally
$$
\vert K(z,w)\vert\preceq\sup_{\sigma}\vert u_{\sigma}(z)\vert_{L_{z}} \preceq e^{-\alpha d_{g}(z,w)}.
$$
\end{proof}
\subsection{Boundedness of the Bergman Projection on $\mathcal{F}^{p}(M,L)$}
The following poposition is a consequence of Theorem 3.6
\begin{prop} Let $(M,g)$ be a K\"ahler Cartan-Hadamard manifold with bounded geometry. such that
$$\sup_{z\in M}\displaystyle\int_{M}e^{-\beta d_{g}(w,z)}dv_{g} < \infty$$
for all $\beta > 0$.
Let $(L,h)\longrightarrow (M,g)$ be a hermitian holomorphic line bundle with bounded curvature such that $$c(L)+Ricci(g)\geq a\omega_{g}$$
for some positive constant $a$. Let $p\in [1,+\infty]$. Then the Bergman projection is bounded as a map from $L^{p}(M,L)$ to $\mathcal{F}^{p}(M,L)$.
\end{prop}
\begin{proof}
 If $p=\infty$, we have
\begin{eqnarray*}
\Vert Ps\Vert_{\infty} &=& \Big\Vert\int_{M}K(z,w).s(w)dv_{g}(w)\Big\Vert_{\infty}\\
&\leq &\Vert s\Vert_{\infty}\sup_{z\in M}\int_{M}\vert K(z,w)\vert dv_{g}(w)\\
&\preceq &\Vert s\Vert_{\infty}\sup_{z\in M}\int_{M}e^{-\alpha d_{g}(z,w)}dv_{g}(w)\\
&\preceq &\Vert s\Vert_{\infty}
\end{eqnarray*}
$P$ is bounded from $L^{\infty}(M,L)$ to $\mathcal{F}^{\infty}(M,L)$.If $ p\in [1,\infty[$,
\begin{eqnarray*}
\int_{M}\vert Ps(z)\vert^{p}dv_{g}(w)&=&\int_{M}\Big\vert\int_{M}K(z,w).s(w)dv_{g}(w)\Big\vert^{p}dv_{g}(z)\\
&\leq &\int_{M}\Big\vert\int_{M}\vert s(w)\vert K(z,w)\vert dv_{g}(w)\Big\vert^{p}dv_{g}(z)\\
 &\leq &\int_{M}\Bigl(\Bigl(\int_{M}\vert K(z,w)\vert dv_{g}(w)\Bigr)^{p-1}\\
&{}&\times\int_{M}\vert s(w)\vert^{p}\vert K(z,w)\vert dv_{g}(w)\Bigr)dv_{g}(z) \hbox{( Jensen inequality)}\\
&\preceq & \int_{M}\Bigl(\int_{M}e^{-\alpha d_{g}(w,z)}dv_{g}(w)\Bigr)^{p-1}\\
&{}&\ \times\int_{M}\vert  s(w)\vert^{p}\vert K(z,w)\vert dv_{g}(w)\Bigr)dv_{g}(z)
\end{eqnarray*}
Thus
\begin{eqnarray*}
\int_{M}\vert Ps(z)\vert^{p}dv_{g}(w)&\preceq &\int_{M}\int_{M}\vert s(w)\vert^{p}e^{-\alpha d_{g}(w,z)}dv_{g}(w)dv_{g}(z)\\
&\preceq & \int_{M}\vert s(w)\vert^{p} dv_{g}(w)
\end{eqnarray*}
and then $P$ is bounded from $L^{p}(M,L)$ to $\mathcal{F}^{p}(M,L)$.
\end{proof}
\section{Boundedness and Compactness for Toeplitz Operators}
Let  $(M,g)$ is a  K\"ahler manifold. Consider the following conditions :\\
(1) $(M,g)$ is a Cartan-Hadamard manifold.\\
(2) $(M,g)$ has bounded geometry,\\
(3) $(L,h)\longrightarrow (M,g)$ is a hermitian holomorphic line bundle with bounded curvature such that
$$
\qquad c(L)+Ricci(g)\geq a\omega_{g}
$$
for some positive constant $a$.\\
(4) For all $\beta > 0$
$$
\sup_{z\in M}\int_{M}e^{-\beta d_{g}(w,z)}dv_{g}(w) < \infty
$$
\begin{rem} Let $(M,g)$ has bounded geometry and $Ricci(g)\geq Kg$. Since
$$
\int_{M}e^{-\beta d_{g}(w,z)}dv_{g}(w)\asymp \int_{0}^{\infty}e^{-\beta r}vol(B_{g}(z,r)) dr
$$
if the volume of $(M,g)$ grows uniformly subexponentially then it saisfies the condition (4). In particular this is true if the volume of $(M,g)$ grows uniformly polynomially.
\end{rem}
 \subsection{Carleson Measures for $\mathcal{F}^{p}(M,L)$}
\begin{defn} A  positive measure $\mu$ on $M$ is  Carleson for $\mathcal{F}^{p}(M,L),\  1\leq p  <\infty,$  if the exists $C_{\mu,p} > 0$ such that
$$
\forall\ s\in\mathcal{F}^{p}(M,L)\ :\ \int_{M}\vert s\vert^{p}d\mu\leq C_{\mu,p}\int_{M}\vert s\vert^{p}dv_{g}
$$
If $p=\infty$, the measure $\mu$ on $M$ is  Carleson for $\mathcal{F}^{\infty}(M,L)$ if there exists $ C, r> 0$ such that $\mu(B(z,r))\leq C$.
\end{defn}
\noindent The following is a geometric characterization of Carleson measures established earlier for  classical Bargman-Fock  space by Ortega Cerda [22] and for generalized Bargman-Fock space by Schuster-Varolin [23].
\begin{thm} Let $(M,g)$ be a K\"ahler manifold which satisfies (1),(2) and (3).
Let $\mu$ be a positive measure on $M$. Let $p\in [1,\infty[$. If $p\not=2$ or $p\not=\infty$, suppose further
$$
\sup_{z\in M}\int_{M}e^{-\beta d_{g}(w,z)}dv_{g}(w) < \infty
$$
for all $\beta > 0$.
The following are equivalent.\\
(a) The measure $\mu$ is Carleson for $\mathcal{F}^{p}(M,L)$.\\
(b) There exists $ r_{0}> 0$ such that $\mu(B_{g}(z,r))\leq C_{r_{0}}$ for any $z\in M$.\\
(c) For each $r > 0$ there exists $C_{r} > 0$ such that $\mu(B_{g}(z,r))\leq C_{r}$ for any $z\in M$.
\end{thm}
\begin{proof}
\noindent $(c)\Longrightarrow (b)$ is trivial. For $(b)\Longrightarrow (c)$, fix  $r > r_{0}$ and an $r_{0}$-lattice  $(a_{k})$ in $M$. There exists an integer $N$ such that each point $z\in M$ has a neighborhood which intersect at most $N$ of the $B_{k}(a_{k},r_{0})$'s. Hence
$$
\mu(B(z,r))\leq\sum_{k=1}^{N}\mu(B_{g}(a_{k},r_{0}))\leq NC_{r_{0}}
$$
\noindent$(b)\Longrightarrow(a)$. Let $s\in\mathcal{F}^{p}(M,L)$. By $(3.1)$ of proposition 3.1
\begin{eqnarray*}
\int_{B_{g}(a_{k},r_{0}/2)}\vert s\vert^{p}d\mu &\leq &\mu(B(a_{k},r_{0}/2))\sup_{w\in B_{g}(a_{k},r_{0}/2)}\vert s(w)\vert^{p}\\
&\preceq &\sup_{w\in B_{g}(a_{k},r_{0}/2)}\vert s(w)\vert^{2}\\
&\preceq &\int_{B_{g}(a_{k},r_{0})}\vert s(w)\vert^{p}dv_{g}
\end{eqnarray*}
Hence
\begin{eqnarray*}
\int_{M}\vert s\vert^{p}d\mu&\preceq &\sum_{k=1}^{\infty}\int_{B_{g}(a_{k},r_{0})}\vert s\vert^{p}d\mu\\
&\preceq &\sum_{k=1}^{\infty}\int_{B_{g}(a_{k},r_{0})}\vert s\vert^{2}dv_{g}\\
&\preceq &\int_{M}\vert s\vert^{p}dv_{g}
\end{eqnarray*}
\noindent $(a)\Longrightarrow (b)$. Let $z\in M$. By proposition 3.3 there is a section  $s_{z}\in\mathcal{F}^{p}(M,L)$ such that
$$
\vert s_{z}(z)\vert=1\quad\hbox{and}\quad\int_{M}\vert s_{z}\vert^{p}dv_{g}\leq C
$$
for some $C > 0$ independant of $z$. Also by lemma 3.2 there exists $0<\delta <R$ such that
$$
\forall\ w\in B_{g}(z,\delta)\ :\ \vert s_{z}(w)\vert\geq {1\over 2}
$$
Hence
\begin{eqnarray*}
\mu(B_{g}(z,\delta)) &\preceq & \int_{B_{g}(z,\delta)}\vert s_{z}\vert^{p}d\mu\\
&\preceq &\int_{M}\vert s_{z}\vert^{p} d\mu\\
&\preceq &\int_{M}\vert s_{z}\vert^{p}dv_{g}\\
&\preceq & 1\quad\hbox{ ( by Carleson condition)}
\end{eqnarray*}
 \end{proof}
\subsection{Vanishing Carleson Measures for $\mathcal{F}^{2}(M,L)$}
\noindent Recall that a bounded linear operator $T : \mathcal{F}^{2}(M,L)\longrightarrow L^{2}(M,L,d\mu)$ is a compact operator  if for all  sequence $(s_{j})\subset\mathcal{F}^{2}(M,L)$ converging weakly to zero section i.e
$$
\forall\ \sigma\in \mathcal{F}^{2}(M,L)\ :\ \lim_{j\rightarrow\infty}\int_{M}<\sigma(w), s_{j}(w)>dv_{g}=0
$$
we have
$$
\lim_{j\rightarrow\infty}\int_{M}\vert Ts_{j}\vert^{p}d\mu=0
$$
The following lemma is a consequence of proposition 3.1, Montel's Theorem and Alaouglu's Theorem.
\begin{lem}
Let $(s_{j})$ be a sequence in $\mathcal{F}^{2}(M,L)$. The following are equivalent.\\
(a) $(s_{j})$ converges weakly zero.\\
(b) There exists $C > 0$ such that
$$
\sup_{j}\int_{M}\vert s_{j}\vert^{2}dv_{g}\leq C
$$
and for all compact $F\subset M$
$$
\lim_{j\rightarrow\infty}\sup_{z\in F}\vert s_{j}(z)-s(z)\vert=0
$$
\end{lem}
\begin{defn} A positive measure $\mu$ on $M$ is a vanishing Carleson if the inclusion $\imath_{\mu} : \mathcal{F}^{2}(M,L)\longrightarrow L^{2}(M,L,\mu)$ is a compact operator.
\end{defn}
\begin{thm} Let $(M,g)$ be a K\"ahler manifold which satisfies (1),(2) and (3). Let $\mu$ be a positive measure on $M$. Then the following are equivalent.\\
(a) The measure $\mu$ is a vanishing Carleson for $\mathcal{F}^{2}(M,L)$.\\
(b) For every $\epsilon > 0$ there exists $r > 0$ such that $\mu(B_{g}(z,R))\leq\epsilon$ for any $z\in M\setminus B_{g}(z_{0},r)$, where $z_{0}\in M$  fixed.
\end{thm}
\begin{proof}
\noindent $(b)\Longrightarrow (a)$ Let $s\in\mathcal{F}^{2}(M,L)$. By proposition $(2.1)$
$$
\vert s(z)\vert\preceq \int_{M}\hbox{\bf 1}_{B_{g}(z,1)}\vert s\vert^{2}dv_{g}
$$
Hence
\begin{eqnarray*}
\int_{M}\vert s(z)\vert^{2}d\mu &\preceq &\int_{M}\int_{M}\hbox{\bf 1}_{B_{g}(z,1)}\vert s(w)\vert^{2}dv_{g}(w) d\mu(z)\\
&=&\int_{M}\vert s(w)\vert^{2}\mu(B(z,1))dv_{g}(w)
\end{eqnarray*}
Let $(s_{j})\subset\mathcal{F}^{2}(M;L)$ be a sequence converging weakly to zero. By lemma 3.4 $(s_{j})$ is bounded by $C$ on $\mathcal{F}^{2}(M,L)$ and converge to zero locally uniformly in $M$. Let $\epsilon > 0$ and $r > 0$ such that $\mu(B_{g}(z,1))<\epsilon$ for $z\in M\setminus B_{g}(z_{0},r)$. For $j$ large enough
\begin{eqnarray*}
\int_{M}\vert s_{j}\vert^{2}d\mu &\preceq &\int_{B_{g}(z_{0},r)}\vert s_{j}(z)\vert^{2}\mu(B_{g}(z,1))dv_{g}(z)\\
&{}& +\epsilon\int_{M\setminus B_{g}(z_{0},r)}\vert s_{j}(z)\vert^{2}\mu(B_{g}(z,1))dv_{g}(z)\\
&\preceq &\int_{B_{g}(z_{0},r)}\vert s_{j}(z)\vert^{2}\mu(B_{g}(z,1))dv_{g}(z)+C\epsilon \\
&\preceq & 2C\epsilon
\end{eqnarray*}
Thus $\mu$ is a vanishing  Carleson measure.\\
\noindent $(a)\Longrightarrow (b)$ Let $(z_{j})\subset M$ such that $d_{g}(z_{j},z_{0})\longrightarrow\infty$. For each $j$, let  $s_{j}\in \mathcal{F}^{2}(M,L)$  such that
$$
\vert s_{j}(w)\vert=\vert K(w,z_{j})\vert\ \hbox{and}\ \Vert s_{j}\Vert_{2}\asymp 1
$$
Then $s_{j}\longrightarrow 0$ locally uniformly in $M$. Since $\mu$ is vanishing Carleson
$$
\lim_{j\rightarrow\infty}\int_{M}\vert s_{j}\vert^{2}d\mu=0
$$
By proposition 3.5 there exist positive constants $C_{1},C_{2}$ and $\delta$ such that
$$\vert K(z,w)\vert\geq C_{1}\vert K(z,z)\vert\geq C_{2}$$
for all $w\in B_{g}(z,\delta)$. Then
\begin{eqnarray*}
\int_{M}\vert s_{j}\vert^{2} d\mu &\geq &\int_{B_{g}(z_{j},\delta)}\vert s_{z_{j}}\vert^{2}d\mu\\
&= & \int_{B_{g}(z_{j},\delta)}\vert K(z,z_{j})\vert^{2}d\mu\\
&\succeq & \mu(B_{g}(z_{j},\delta))\vert K(z_{j},z_{j})\vert^{2}\\
&\succeq & \mu(B_{g}(z_{j},\delta))
\end{eqnarray*}
since $\vert K(z_{j},z_{j})\vert\asymp 1$ uniformly in $j$. Hence
$$\lim_{j\rightarrow\infty}\mu(B_{g}(z_{j},\delta))=0$$  Since $B_{g}(z_{j},1)$ is covered by $N$ balls $B_{g}(a_{k_{1}},\delta),\cdots, B_{g}(a_{k_{N}},\delta)$ ( $\delta$-lattice ), it follows that
$$
\lim_{j\rightarrow\infty}\mu(B_{g}(z_{j},1))=0
$$
\end{proof}
\subsection{Berezin Transforms of Carleson Measures}
Let $\mu$ be a positive meaure on $M$. The Berezin transform of $\mu$ is the function $\tilde{\mu} : M\rightarrow \mathbb{R}^{+}$ defined by
$$
\tilde{\mu}(z):=\int_{M}\vert k_{z}(w)\vert^{2}d\mu(w)
$$
where
$$
k_{z}(w):={K(w,z)\over\sqrt{\vert K(z,z)\vert}}
$$
\begin{thm} Let $(M,g)$ be a K\"ahler manifold satisfying the conditions (1),(2) and (3). Let $\mu$ be a positive measure on $M$. Let
$p\in [1,\infty]$. If $p\not=2$ or $p\not=\infty$ suppose further
$$
\sup_{z\in M}\int_{M}e^{-\beta d_{g}(w,z)}dv_{g}(w) < \infty
$$
for all $\beta > 0$.
The following are equivalent.\\
(a) $\mu$ is Carleson for $\mathcal{F}^{p}(M,L)$.\\
(b) $\tilde{\mu}$ is bounded on $M$.
\end{thm}
\begin{proof}
\noindent$(a)\Longrightarrow(b)$ For $z\in M$ let $s_{z}\in\mathcal{F}^{2}(M,L)$ such that $\vert s_{z}(w)\vert=\vert K(z,w)\vert$. By off-diagonal estimate $\vert s_{z}(w)\vert\leq Ce^{-\alpha d_{g}(z,w)}\preceq 1$.
Let $(a_{i})$ be a lattice  of $M$.  Since $\mu$ is Carleson by Theorem 4.1  $\mu(B_{g}(a_{j},r))\leq C$. We have
\begin{eqnarray*}
\tilde{\mu}(z)&=&{1\over\sqrt{\vert K(z,z)\vert}}\int_{M}\vert s_{z}\vert^{2}d\mu(w)\\
&\leq &\sum_{j}\int_{B_{g}(a_{j},r)}\vert s_{z}\vert^{2}d\mu(w)\ \hbox{( since $\vert K(z,z)\vert\asymp 1$)}\\
&\leq &\sum_{j}\Bigl(\int_{B_{g}(a_{j},r)}\vert s_{z}\vert^{p}d\mu(w)\Bigr)^{1\over p}\Bigl(\int_{B_{g}(a_{j},r)}\vert s_{z}\vert^{q}d\mu(w)\Bigr)^{1\over q}\\
&\leq &\sum_{j}\Bigl(\int_{B_{g}(a_{j},r)}\vert s_{z}\vert^{p}d\mu(w)\Bigr)^{1\over p}\mu(B_{g}(a_{j},r))^{1\over q}\sup_{B_{g}(a_{j},r)}\vert s_{z}(w)\vert\\
&\preceq &\Bigl(\int_{M}\vert s_{z}\vert^{p}d\mu(w)\Bigr)^{1\over p}\\
&\preceq & \Bigl(\int_{M}\vert s_{z}\vert^{p}dv_{g}(w)\Bigr)^{1\over p}\quad \hbox{($\mu$ is Carleson for $\mathcal{F}^{p}(M,L)$)}\\
&\preceq &\Bigl(\int_{M}\vert s_{z}\vert^{p}dv_{g}(w)\Bigr)^{1\over p}\\
&\preceq & 1
\end{eqnarray*}
  Hence if $\mu$ is a Carleson  then $\tilde{\mu}$ is uniformly bounded. \\
\noindent$(b)\Longrightarrow(a)$ Suppose that $\tilde{\mu}$ is bounded on $M$. Then there exists $C > 0$ such that for all $\delta > 0$ and $z\in M$
$$
\int_{B_{g}(z,\delta)}\vert k_{z}(w)\vert^{2}d\mu(w)\leq\tilde{\mu}(z)\leq C
$$
By diagonal estimates for the Bergman Kernel there exists $C_{1},\delta > 0$ independent of $z$ such that for all $w\in B_{g}(z,\delta)$
$$
\vert K(z,w)\vert\geq C_{1}\vert K(z,z)\vert
$$
Since $\vert K(z,z)\vert\asymp 1$
$$
\vert k_{z}(w)\vert^{2} \succeq 1,\ \forall\ w\in B_{g}(z,\delta)
$$
Hence
$$
\mu(B_{g}(z,\delta))\preceq 1\quad\hbox{uniformly for $z\in M$}
$$
and by Theorem 4.3 $\mu$ is Carleson for $\mathcal{F}^{p}(M,L)$.
\end{proof}
\subsection{Berezin Transforms of Vanishing Carleson Measures}
\begin{thm}
Let $(M,g)$ be a K\"ahler manifold satisfying the conditions (1),(2) and (3). Let $\mu$ be a positive measure on $M$. The following are equivalent.\\
(a) $\mu$ is vanishing Carleson for $\mathcal{F}^{2}(M,L)$.\\
(b) $\displaystyle\lim_{d_{g}(z,z_{0})\rightarrow \infty}\tilde{\mu}(z)=0$.
\end{thm}
\begin{proof}
\noindent$(a)\Longrightarrow(b)$ Let $(z_{n})\in M$ such that $\lim_{n\rightarrow\infty}d_{g}(z_{n},z_{0})=\infty$. For $n\in\mathbb{N}$ let $s_{n}\in\mathcal{F}^{2}(M,L)$ such that $\vert s_{n}(w)\vert=\vert K(w,z_{n})\vert$. Put
$$
\tilde{s}_{n}(w)={s_{n}(w)\over\sqrt{\vert K(z_{n},z_{n})\vert}}
$$
Then $\tilde{s}_{n}\in \mathcal{F}^{2}(M,L)$. Since $\vert K(z_{n},z_{n})\vert \asymp 1$ uniformly in $n$ and
$$
\vert\tilde{s}_{n}(w)\vert\leq Ce^{-\alpha d_{g}(w,z_{n})}
$$
then $\displaystyle\lim_{n\rightarrow\infty}\tilde{s}_{n}(w)=0$ and
$$
\int_{M}\vert\tilde{s}_{n}\vert^{2}dv_{g}(w)=1
$$
So $\tilde{s}_{n}\rightarrow 0$ uniformly on compacts of $M$. By lemma 4.4 $\tilde{s}_{n}\rightarrow 0$ weakly on $\mathcal{F}^{2}(M,L)$. Since $\mu$ is vanishing Carleson
$$
\lim_{n\rightarrow\infty}\tilde{\mu}(z_{n})=\lim_{n\rightarrow\infty}\int_{M}\vert\tilde{s}_{n}(w)\vert^{2}d\mu(w)=0
$$
\noindent$(b)\Longrightarrow(a)$ Following the proof of $(b)\Longrightarrow(a)$ in Theorem 4.3 we have
$$
\mu(B(z,r))\preceq \tilde{\mu}(z)
$$
Hence
$$
\lim_{d_{g}(z,z_{0})\rightarrow\infty}\mu(B(z,r))\preceq \lim_{d_{g}(z,z_{0})\rightarrow\infty}\tilde{\mu}(z)=0
$$
By Theorem 4.6 $\mu$ is vanishing Carleson.
\end{proof}
\subsection{Proof ot Theorem 1.2}
$(b)\Longleftrightarrow(c)$ follows from Theorem 4.7\\
$(b)\Longleftrightarrow(d)$ follows from Theorem 4.3.\\
$(b)\Longleftrightarrow(a)$ Suppose that $\mu$ is a Carleson measure. Fix $p\in ]1,\infty[$. Let $s\in \mathcal{F}^{p}(M,L)$. Then
$$
\int_{M}\Big\vert\int_{M}<s(w),K(w,z> d\mu(w) \Big\vert^{p}dv_{g}(z)\qquad\qquad\qquad\qquad\qquad\qquad\quad{}$$
\begin{eqnarray*}
{}&\leq &\int_{M}\Bigl(\int_{M}\vert s\vert\vert K(w,z)\vert  d\mu(w) \Bigr)^{p}dv_{g}(z)\\
{}&\leq &\int_{M}\Bigl(\int_{M}\vert s\vert\vert K(w,z)\vert^{1\over p}\vert K(w,z)\vert^{1\over q} d\mu(w) \Bigr)^{p}dv_{g}(z){}\\
{}&\leq & \int_{M}\Bigl(\int_{M}\vert s(w)\vert^{p}\vert K(w,z)\vert  d\mu(w) \Bigr)\Bigl(\int_{M}\vert K(z,w)\vert d\mu(w)\Bigr)^{p-1}dv_{g}(z)
\end{eqnarray*}
Let $s_{z}\in\mathcal{F}^{2}(M,L)$ such that $\vert s_{z}(w)\vert=\vert K(w,z)\vert$. Then
\begin{eqnarray*}
\int_{M}\vert K(w,z)\vert d\mu(w)&=&
\int_{M}\vert s_{z}(w)\vert d\mu(w)\\
&\preceq & \int_{M}\vert s_{z}(w)\vert dv_{g}(w)\ \hbox{($\mu$ is Carleson for $\mathcal{F}^{1}(M,L)$)}\\
&=& \int_{M}\vert K(w,z)\vert dv_{g}(w)\\
&\leq & C\int_{M}e^{-\alpha d_{g}(w,z)}dv_{g}(w)\preceq 1
\end{eqnarray*}
and
$$
\int_{M}\Bigl(\int_{M}\vert s(w)\vert^{2}\vert K(w,z)\vert d\mu(w)\Bigr)dv_{g}(z)
\qquad\qquad\qquad\qquad\qquad\quad{}$$
\begin{eqnarray*}
{}&\leq &\int_{M}\vert s(w)\vert^{p}\Bigl(\int_{M}\vert K(w,z)\vert dv_{g}(z)\Bigr)d\mu(w)\\
{}&\preceq &\int_{M}\vert s\vert^{p}d\mu(w)\ \hbox{(by off-diagonal estimate)}\\
{}&\preceq & \int_{M}\vert s\vert^{p} dv_{g}\ \hbox{($\mu$ is Carleson for $\mathcal{F}^{p}(M,L)$)}
\end{eqnarray*}
Hence
$$
\int_{M}\vert T_{\mu}s(w)\vert^{p}dv_{g}(z)\leq C_{\mu}\int_{M}\vert s\vert^{p} dv_{g}
$$
If $f\in\mathcal{F}^{1}(M,L)$ then \\
$\displaystyle\int_{M}\Big\vert\int_{M}<s(w),K(w,z> d\mu(w)\Big\vert dv_{g}(z)$\\
${}\qquad\qquad\qquad\qquad\leq\displaystyle\int_{M}\Bigl(\int_{M}\vert s\vert\vert K(w,z)\vert  d\mu(w) \Bigr)dv_{g}(z)$\\
${}\qquad\qquad\qquad\qquad\leq\displaystyle\int_{M}\Bigl(\int_{M}\vert s\vert\vert K(w,z)\vert d\mu(w)\Bigr)dv_{g}(z)$\\
${}\qquad\qquad\qquad\qquad\leq\displaystyle \int_{M}\vert s(w)\vert\Bigl(\int_{M}\vert K(z,w)\vert dv_{g}(z)\Bigr)d\mu(w)$\\
${}\qquad\qquad\qquad\qquad\leq\displaystyle\int_{M}\vert s(w)\vert\Bigl(\int_{M}e^{-\alpha d_{g}(z,w)}dv_g(z)\Bigr)d\mu(w)$\\
${}\qquad\qquad\qquad\qquad\leq\displaystyle\int_{M}\vert s(w)\vert d\mu(w)$\\
${}\qquad\qquad\qquad\qquad\leq \displaystyle\int_{M}\vert s(w)\vert dv_{g}(w)\ \hbox{($\mu$ is Carleson for $\mathcal{F}^{1}(M,L)$)}$.\\
Hence
$$
\int_{M}\vert  T_{\mu}s(w)\vert dv_{g}(z)\leq C_{\mu}\int_{M}\vert s\vert  dv_{g}
$$
If $f\in\mathcal{F}^{\infty}(M,L)$ then
\begin{eqnarray*}
\sup_{z\in M}\Big\vert\int_{M}<s(w),K(w,z>d\mu\Big\vert&\leq &\Vert s\Vert_{\infty}\sup_{z\in M}\int_{M}\vert K(z,w)\vert d\mu(w)\\
&=&\Vert s\Vert_{\infty}\sup_{z\in M}\int_{M}\vert s_{z}(w)\vert d\mu(w)\\
&\preceq &\Vert s\Vert_{\infty}\sup_{z\in M}\int_{M}\vert s_{z}(w)\vert dv_{g}(w)\\
&\preceq &\Vert s\Vert_{\infty}\sup_{z\in M}\int_{M}\vert K(z,w)\vert dv_{g}(w)\\
&\preceq &\Vert s\Vert_{\infty}\sup_{z\in M}\int_{M}e^{-\alpha d_{g}(z,w)} dv_{g}(w)\\
&\preceq &\Vert s\Vert_{\infty}
\end{eqnarray*}
Hence
$$
\sup_{z\in M}\vert T_{\mu}s(z)\vert \leq C_{\mu}\sup_{z\in M}\vert s(z)\vert
$$
We conclude that $T_{\mu} : \mathcal{F}^{p}(M,L)\rightarrow \mathcal{F}^{p}(M,L)$ is well defined and bounded if $\mu$ is Carleson.\\
Inversely, suppose $T_{\mu} : \mathcal{F}^{p}(M,L)\rightarrow \mathcal{F}^{p}(M,L)$ is bounded. Let $s_{z}\in \mathcal{F}^{2}(M,L)$ such that $\vert s_{z}(w)\vert=\vert K(w,z)\vert$. By reproducing property of the Bergman kernel
$$
s_{z}(w)=\int_{M}<s_{z}(t),K(t,w)>dv_{g}(t)
$$
By diagonal bounds for the Bergman kernel , there exists $C, \delta > 0$ such that $
\vert s_{z}(w)\vert\geq C$ for all $w\in B_{g}(z,\delta)$. We have
\begin{eqnarray*}
\mu(B_{g}(z,\delta))&\preceq &\int_{B_{g}(z,\delta)}\vert s_{z}(w)\vert^{2}d\mu(w)\\
&\preceq &\int_{M}\vert s_{z}(w)\vert^{2}d\mu(w)\\
&=& \int_{M}<s_{z}(w),\int_{M}<s_{z}(t),K(t,w)>dv_{g}(t)>d\mu(w)\\
&=&\int_{M}\Bigl(\int_{M}<s_{z}(w),<s_{z}(t),K(t,w)>>d\mu(w)\Bigr)dv_{g}(t)\\
&=&\int_{M}\Bigl(\int_{M}<s_{z}(t),<s_{z}(w),K(w,t)>>d\mu(w)\Bigr)dv_{g}(t)\\
&=&\int_{M}<s_{z}(t),\int_{M}<s_{z}(w),K(w,t)>d\mu(w)>dv_{g}(t)\\
&=&\int_{M}<s_{z}(t),T_{\mu}s_{z}(t)>dv_{g}(t)\\
&\leq &\Vert T_{\mu}s_{z}\Vert_{p}\Vert s_{z}\Vert_{q}\leq \Vert T_{\mu}\Vert\Vert s_{z}\Vert_{p}\Vert s_{z}\Vert_{q}\leq C
\end{eqnarray*}
Therfore by Theorem 4.3 $\mu$ is Carleson for $\mathcal{F}^{p}(M,L)$.
\subsection{Proof ot Theorem 1.3}
$(b)\Longleftrightarrow(c)$ follows from Theorem 4.8.\\
$(b)\Longleftrightarrow(d)$ follows from Theorem 4.6.\\
$(b)\Longleftrightarrow(a)$ Suppose that $\mu$ is vaninshing  Carleson. Let $s\in \mathcal{F}^{2}(M,L)$. Let $s_{z}$ the holomrphic section such that $\vert s_{z}(w)\vert=\vert K(w,z)\vert$. Then
$$
\int_{M}\vert T_{\mu}(z)\vert^{2}dv_{g}(z)=\int_{M}\Big\vert\int_{M}<s(w),K(w,z)>d\mu(w)\Big\vert^{2}dv_{g}(z)
$$
\begin{eqnarray*}
&\leq &\int_{M}\Bigl(\int_{M}\vert s(w)\vert^{2}\vert K(w,z)\vert d\mu(w)\Bigr)\Bigl(\int_{M}\vert K(w,z)\vert d\mu(w)\Bigr)dv_{g}(z)\\
&= &\int_{M}\Bigl(\int_{M}\vert s(w)\vert^{2}\vert K(w,z)\vert d\mu(w)\Bigr)\Bigl(\int_{M}\vert s_{z}(w)\vert d\mu(w)\Bigr)dv_{g}(z)\\
&\preceq &\int_{M}\Bigl(\int_{M}\vert s(w)\vert^{2}\vert K(w,z)\vert d\mu(w)\Bigr)\Bigl(\int_{M}\vert s_{z}(w)\vert dv_{g}(w)\Bigr)dv_{g}(z)\\
&\preceq &\int_{M}\int_{M}\vert s(w)\vert^{2}\vert K(w,z)\vert d\mu(w)dv_{g}(z)\Bigl(\sup_{z\in M}\int_{M}\vert s_{z}(w)\vert dv_{g}(w)\Bigr)\\
&\preceq &\int_{M}\int_{M}\vert s(w)\vert^{2}\vert K(z,w)\vert dv_{g}(z)d\mu(w) \\
&\preceq &\int_{M}\vert s(w)\vert^{2}d\mu(w)
\end{eqnarray*}
Hence $\Vert T_{\mu}\Vert\leq C\Vert \imath_{\mu}\Vert$ and this follows that $T_{\mu}$ is compact.\\
Inversely suppose that $T_{\mu} : \mathcal{F}^{2}(M,L)\rightarrow \mathcal{F}^{2}(M,L)$ is compact. Let $(z_{j})\in M$ such that $d_{g}(z_{j},z_{0})\rightarrow 0$ and $s_{z_{j}}\in \mathcal{F}^{2}(M,L)$ such that $\vert s_{z_{j}}(w)\vert=\vert K(w,z_{n})\vert$. By off-diagonal estimate, the sequence $(s_{z_{j}})$ is bounded on $\mathcal{F}^{2}(M,L)$ and converge locally uniformly to zero section. Hence $(s_{z_{j}})$  converge weakly to the zero. Since $T_{\mu}$ is compact and
$$
\Big\vert\int_{M}<T_{\mu}s_{z_{j}},s_{z_{j}}>dv_{g}\Big\vert\leq\Vert T_{\mu}s_{z_{j}}\Vert_{2}\Vert s_{z_{j}}\Vert_{2}
$$
we have
$$
\lim_{j\rightarrow\infty}\int_{M}<T_{\mu}s_{z_{j}},s_{z_{j}}>dv_{g}=0
$$
From
$$
\Big\vert\int_{M}<T_{\mu}s_{z_{j}},s_{z_{j}}>dv_{g}\Big\vert=\int_{M}\vert s_{z_ {j}}\vert^{2}dv_{g}
$$
the diagonal estimates  $\vert s_{z_{j}}(w)\vert \succeq 1$ on $B_{g}(z_{j},\delta)$ , we get
$$
\lim_{j\rightarrow\infty}\mu(B_{g}(z_{j},\delta))\preceq \lim_{j\rightarrow\infty}\Big\vert\int_{M}<T_{\mu}s_{z_{j}},s_{z_{j}}>
dv_{g}\Big\vert=0
$$
By Theorem 4.8 $\mu$ is vanishing Carleson for $\mathcal{F}^{2}(M,L)$.
\section{Schatten Class Memberhip of Toeplitz Operators}
Suppose that $T$ is a compact operator between
Hilbert spaces $H_{1}$ and $H_{2}$. Then $T$ has a Schmidt decomposition, so that
there are orthonormal bases $(e_{n})$ and $(\sigma_{n})$ of $H_{1}$ and $H_{2}$ respectively and a sequence
$(\lambda_{n})$ with $\lambda_{n} > 0$ and $\lambda_{n}\rightarrow 0$ such that for all $f\in H_{1}$
$$
Tf=\sum_{n=0}^{\infty}\lambda_{n}<f,e_{n}>\sigma_{n}
$$
For $0 < p \leq \infty$, such a compact operator $T$ belongs to the Schatten-von Neumann $p$-class $\mathcal{S}_{p}=\mathcal{S}_{p}(H_{1},H_{2})$ if and only if
$$
\Vert T\Vert_{\mathcal{S}_{p}}^{p}:=\sum_{n=0}^{\infty}\lambda_{n}^{p} < \infty
$$
If $p\geq 1$ then $\mathcal{S}_{p}$ is a Banach space. If $0 < p < 1$ then $\mathcal{S}_{p}$  is  a Frechet space.\\  For all $T,S \in \mathcal{S}_{p}(H_{1},H_{1})$.
\begin{equation}
\Vert T+S\Vert_{\mathcal{S}_{p}}^{p}\leq 2(\Vert T\Vert^{p}_{\mathcal{S}_{p}}+\Vert S\Vert_{\mathcal{S}_{p}}^{p})
\end{equation}
By  Proposition 6.3.3 in [33], if $T$ is a positive operator on a Hilbert space $H$ and  $0 < p< 1$ then
$$
<T^{p}e_{m},e_{m}>\le <Te_{m},e_{m}>^{p}
$$
where $(e_{m})$ is an orthonormal set of $H$.
It give that
$$
\Vert T\Vert^{p}_{\mathcal{S}_{p}}\le\sum_{m,k}^{\infty}\vert<Te_{m},e_{k}>\vert^{p}
$$
We will introduce the complex interpolation of Schatten $p$-class.
\begin{lem}
If $1\leq p \leq \infty$ then
$$
[\mathcal{S}_{p_{0}},\mathcal{S}_{p_{1}}]_{\theta}=\mathcal{S}_{p}
$$
with equal norm for all $1\leq p_{0} < p_{1} \leq\infty$ and all $\theta\in ]0,1[$, where
$$
{1\over p}={1-\theta\over p_{0}}+{\theta\over p_{1}}
$$
\end{lem}
\noindent  We will let $(a_{j})$ denote an $r$-lattice of $M$ and $\tilde{\mu}$ the Berezin transform of the positive measure on $M$. For $z\in M$ let $s_{z}\in\mathcal{F}^{2}(M,L)$ such that
$$
s_{z}(w)\otimes\overline{e(z)\over\vert e(z)\vert}=K(w,z)
$$
where $e$ is a frame of $L$ around $z$.
\begin{lem} If $T$ is a positive operator on $\mathcal{F}^{2}(M,L)$, then
\begin{eqnarray*}
\hbox{tr}(T) \asymp \int_{M}\tilde{T}(z)dv_{g}(z)
\end{eqnarray*}
where $$\tilde{T}(z) = \int_{M}<Ts_{z}(w),s_{z}(w)>dv_{g}(w)$$
is the Berezin transform of T. In particular, T is trace-class if and only if
the integral above converges.
\end{lem}
\begin{proof} Since $T$ is positive then $T=R^{2}$ for some $R\geq 0$. Let $(e_{j})$ is an orhonormal basis of $\mathcal{F}^{2}(M,L)$. Then
\begin{eqnarray*}
\hbox{tr}(T)&=&\sum_{j=1}^{\infty}<Te_{j},e_{j}>
\asymp \sum_{j=1}^{\infty}\Vert Re_{j}\Vert^{2}\\
&=&\sum_{j=1}^{\infty}\int_{M}\vert Re_{j}(z)\vert^{2}dv_{g}(z)\\
&=& \int_{M}\sum_{j=1}^{\infty}\vert Re_{j}(z)\vert^{2}dv_{g}(z)
\end{eqnarray*}
Hence
\begin{eqnarray*}
\hbox{tr}(T)&=& \int_{M}\sum_{j=1}^{\infty}\Big\vert\int_{M} < Re_{j}(w),K(w,z)>dv_{g}(w)\Big\vert^{2}dv_{g}(z)\\
&=& \int_{M}\sum_{j=1}^{\infty}\Big\vert\int_{M}< Re_{j}(w),s_{z}(w)\otimes\overline{e(z)\over\vert e(z)\vert}>dv_{g}(w)\Big\vert^{2}dv_{g}(z)\\
&=& \int_{M}\sum_{j=1}^{\infty}\Big\vert\int_{M}
< Re_{j}(w),s_{z}(w)>\overline{e(z)\over\vert e(z)\vert}dv_{g}(w)\Big\vert^{2}dv_{g}(z)\\
&=& \int_{M}\sum_{j=1}^{\infty}\Big\vert\int_{M}
< Re_{j}(w),s_{z}(w)>dv_{g}(w)\Big\vert^{2}dv_{g}(z)\\
&=& \int_{M}\sum_{j=1}^{\infty}\Big\vert\int_{M}
< e_{j}(w),Rs_{z}(w)>dv_{g}(w)\Big\vert^{2}dv_{g}(z)\\
&\asymp & \int_{M}\Vert Rs_{z}\Vert^{2}dv_{g}(z)\asymp\int_{M}<Ts_{z},s_{z}>dv_{g}(z)= \int_{M}\tilde{T}(z)dv_{g}(z)
\end{eqnarray*}
\end{proof}
\begin{cor}
Let $\nu$ is a positive measure on $M$. Then $T_{\nu}\in\mathcal{S}_{1}$ if and only if $\mu(M) < \infty$. In particular, if the support of $\mu$ is compact then $T_{\mu}\in\mathcal{S}_{p}$ for each $p\geq 1$.
\end{cor}
\begin{proof}
 Suppose that $\mu(M) < \infty$. By Lemma 5.2
\begin{eqnarray*}
\hbox{tr}(T_{\mu}) &=&\int_{M}\tilde{T_{\mu}}(z)dv_{g}(z)\\
&\asymp & \int_{M}\int_{M}<T_{\mu}s_{z}(w),s_{z}(w)>dv_{g}(w)dv_{g}(z)\\
&\asymp & \int_{M}\int_{M}\vert s_{z}(w)\vert^{2}d\mu(w)dv_{g}(z)\\
&\asymp & \int_{M}\int_{M}\vert K(w,z)\vert^{2}\vert dv_{g}(z)\Bigr)d\mu(w)\\
&\asymp & \int_{M}\vert K(w,w)\vert d\mu(w)\asymp\mu(M)
\end{eqnarray*}
Let $T_{\mu}\in\mathcal{S}_{1}$ and  $z_{0}\in M$ fixed. By diagonal bound estimates we have
\begin{eqnarray*}
\hbox{tr}(T) &\asymp& \int_{M}\tilde{T}(z)dv_{g}(z)
 \asymp  \int_{M}\Bigl(\int_{M}\vert K(w,z)\vert^{2} dv_{g}(z)\Bigr)d\mu(w)\\
&\succeq & \int_{M}\Bigl(\int_{B_{g}(z_{0},\delta)}\vert K(w,z)\vert^{2}\vert dv_{g}(w)\Bigr)d\mu(z)
 \asymp  vol_{g}(B_{g}(z_{0},\delta)\mu(M)\\
&\succeq & \mu(M)
\end{eqnarray*}
\end{proof}
\noindent We will need the following simple lemma that is well known in the
classical Fock space setting [33].
\begin{lem}
Let $r >0$ and let $(e_{j})$ be any orthonormal basis for $\mathcal{F}^{2}(M,L)$. If
$(a_{j})$ is an $r$-latice of $M$ and $H$ is the operator on $\mathcal{F}^{2}(M,L)$
 defined by $He_{j}:= s_{a_{j}}$
 then $H$
extends to a bounded operator on all of $\mathcal{F}^{2}(M,L)$
 whose operator norm is bounded
above by a constant that only depends on $r$.
\end{lem}
\begin{proof}
Let  $\sigma, t\in\mathcal{F}^{2}(M,L)$ then $$<H\sigma,t>=\sum_{j=1}^{\infty}<\sigma,e_{j}><s_{a_{j}},t>$$
Since
$$
s_{a_{j}}(w)\otimes{e(a_{j})\over\vert e(a_{j})\vert}=K(w,a_{j})
$$
where $e$ is a frame of $L$ around $a_{j}$. Since $$t(a_{j})=\int\int_{M} <t(w),K(w,a_{j})> dv_{g}(w)$$
by Cauchy-Schwarz inequality and Proposition 3.1
\begin{eqnarray*}
\vert<A\sigma,t>\vert &\leq &\sum_{j=1}^{\infty}\vert<\sigma,e_{j}>_{L^{2}}\vert\vert<s_{a_{j}},t>_{L^{2}}\vert\\
&=& \sum_{j=1}^{\infty}\vert<\sigma,e_{j}>_{L^{2}}\vert\Big\vert<s_{a_{j}},t>_{L^{2}} {e(a_{j})\over\vert e(a_{j})\vert}\Big\vert_{L_{a_{j}}}\\
&\leq & \Vert\sigma\Vert_{2}\Bigl( \sum_{j=1}^{\infty}\Big\vert\int_{M} <t(w),K(w,a_{j})> dv_{g}(w)\Big\vert^{2}_{L_{a_{j}}}\Bigr)^{1\over 2}\\
&\leq &\Vert\sigma\Vert_{2}\Bigl(\sum_{j=1}^{\infty}\vert t(a_{j})\vert^{2}\Bigr)^{1\over 2}\\
&\preceq &\Vert \sigma\Vert_{2}\Bigl(\sum_{j=1}^{\infty}\int_{B_{g}(a_{j},r)}\vert t\vert^{2}dv_{g}\Bigr)^{1\over 2}\\
&\preceq &\Vert\sigma\Vert_{2}\Vert t\Vert_{2}
\end{eqnarray*}
\end{proof}
\begin{lem}  Let $p\geq 1$. If  $\phi\in L^{p}(M,dv_{g})$ and $T_{\phi}$ the Toeplitz operator with smbol $\phi$
$$
T_{\phi}s(z)=\int_{M}<s(w),K(w,z)>\phi(w)dv_{g}(w)
$$
for all $s\in\mathcal{F}^{2}(M,L)$ then  $T_{\phi}\in\mathcal{S}_{p}$.
\end{lem}
\begin{proof}
Assume $p=1$. Let $g\in L^{1}(M,dv_{g})$ and $(e_{j})$ be an orthonormal set on $\mathcal{F}^{2}(M,L)$. By Fubini theorem
$$
<T_{\phi}e_{j}(z),e_{j}(z)>=\int_{M}\vert e_{j}(z)\vert^{2}\phi(z)dv_{g}(z)
$$
Hence
\begin{eqnarray*}
\sum_{j=1}^{\infty}\vert <T_{\phi}e_{j},e_{j}>\vert&=&\sum_{j=1}^{\infty}\Big\vert\int_{M}\vert e_{j}(w)\vert^{2}\phi(w)dv_{g}(w)\\
&\leq & \int_{M}\sum_{j=1}^{\infty}\vert e_{j}(w)\vert\vert \phi(w)\vert dv_{g}(w)\\
&=&\int_{m}\vert\phi(w)\vert\vert K(w,w)\vert d_{g}(w)\\
&\preceq &\Vert \phi\Vert_{1}\quad\hbox{(by diagonal estimate)}
\end{eqnarray*}
Thus, for $p=1$, $T_{\phi}\in\mathcal{S}_{1}$ and $\Vert T_{\mu}\Vert_{\mathcal{S}_{1}}\preceq \Vert \phi\Vert_{1}$.
Also $\Vert T_{\mu}\Vert_{\mathcal{S}_{\infty}}\preceq \Vert \phi\Vert_{\infty}
$. By interpolation of Lemma 5.1, we can get $T_{\phi}\in\mathcal{S}_{p}$ and
$
\Vert T_{\phi}\Vert_{\mathcal{S}_{p}}\leq \Vert \phi\Vert_{p}.
$
\end{proof}
\begin{lem} Suppose that $(M,g)$ satisfies the conditions (1)(2),(3) and (4) of section 4. Let $ r > 0$ and $0 < p < 1$. The following are equivalent : \\
(a) $\tilde{\mu} \in L^{p}(M,dv_{g})$\\
(b) $\mu(B_{g}(.,r)\in L^{p}(M,dv_{g})$\\
(c) $\mu(B_{g}(a_{j},r))\in  \ell^{p}(\mathbb{N})$
\end{lem}
\begin{proof}
\noindent$(c)\Longrightarrow(a)$  We have
\begin{eqnarray*}
\tilde{\mu}(z)&=&\int_{M}\vert k_{z}(w)\vert^{2}d\mu(w)\\
&=&\sum_{j=1}^{\infty}\int_{B_{g}(a_{j},r)}{\vert K(w,z)\vert^{2}\over\vert K(z,z)\vert}d\mu(w)\\
&\leq & C\sum_{j=1}^{\infty}\int_{B_{g}(a_{j},r)}e^{-2\alpha d_{g}(w,z)}d\mu(w)
\end{eqnarray*}
Since $d_{g}(w,)\geq d_{g}(z,a_{j})-d_{g}(a_{j},w)$ for all $w\in B_{g}(a_{j},r)$
\begin{eqnarray*}
\tilde{\mu}(z) &\leq &\sum_{j=1}^{\infty}\int_{B_{g}(a_{j},r)}e^{-2\alpha d_{g}(w,z)}d\mu(w)\\
&\leq & C\sum_{j=1}^{\infty}\int_{B_{g}(a_{j},r)}e^{-2\alpha( d_{g}(z,a_{j})-r)}d\mu(w)\\
&\preceq & \sum_{j=1}^{\infty}e^{-2\alpha d_{g}(z,a_{j})}\mu(B_{j}(a_{j},r))
\end{eqnarray*}
By H\"older inequality $
\tilde{\mu}(z)^{p} \preceq\displaystyle\sum_{j=1}^{\infty}e^{-2p\alpha d_{g}(z,a_{j})}\mu(B_{j}(a_{j},r))^{p}$.
Hence
\begin{eqnarray*}
\int_{M}\tilde{\mu}(z)^{p} &\preceq &\sum_{j=1}^{\infty}\int_{M}e^{-2p\alpha d_{g}(z,a_{j})}\mu(B_{j}(a_{j},r))^{p}\\
&\preceq &\sum_{j=1}^{\infty}\mu(B_{j}(a_{j},r))^{p}\sup_{j\in\mathbb{N}}\int_{M}e^{-2p\alpha d_{g}(z,a_{j})}\\
&\preceq & \sum_{j=1}^{\infty}\mu(B_{j}(a_{j},r))^{p} < \infty
\end{eqnarray*}
\noindent $(a)\Longrightarrow(b)$ By diagonal bound estimate $\vert K(z,z)\vert\asymp 1$ and $\vert K(z,w)\vert \succeq \vert K(z,z)\vert$ for all $w\in B_{g}(z,\delta)$
\begin{eqnarray*}
\tilde{\mu}(z)&=&\int_{M}\vert k_{z}(w)\vert^{2}d\mu(w)\\
&\geq & \int_{B_{g}(z,r)}\vert k_{z}(w)\vert^{2}d\mu(w)\\
& \succeq & \int_{B_{g}(z,r)}\vert K(w,z)\vert^{2}d\mu(w)\\
&\succeq &\sum_{j=1}^{\infty}\int_{B_{g}(z,r)\cap B_{g}(a_{j},\delta)}\vert K(w,z)\vert^{2}d\mu(w)\\
&\succeq &\sum_{j=1}^{\infty}\int_{B_{g}(z,r)\cap B_{g}(a_{j},\delta)}d\mu(w)
\succeq \mu(B(z,r))
\end{eqnarray*}
\noindent$(b)\Longrightarrow(c)$ We have
$$
\sum_{j=1}^{\infty}\int_{B_{g}(a_{j},{r\over 2})}\mu(B(z,r))^{p}dv_{g}(z) \preceq \int_{M}\mu(B_{g}(z,r))^{p}dv_{g}(z)
$$
Since for any $z\in B_{g}(a_{j},{r\over 2})\ :\ \mu(B_{g}(z,r))\geq\mu(B_{g}(a_{j},{r\over 2})$, then
$$
\sum_{j=1}^{\infty}\mu(B_{g}(a_{j},{r\over 2}))^{p} \preceq\int_{M}\mu(B_{g}(z,r))^{p}dv_{g}(z)
$$
Thus $\mu(B_{g}(.,r))\in L^{p}(M,dv_{g})$ implies that $(\mu(B_{g}(a_{j},r))\in \ell^{p}(\mathbb{N})$.
\end{proof}
\subsection{Proof of Theorem 1.4 for the case $1 \leq p < \infty$}
$(a)\Longrightarrow(b)$.  Since $T_{\mu}$ is a positive operator then $T_{\mu}\in\mathcal{S}_{p}$ if and only if $T_{\mu}^{p}\in\mathcal{S}_{1}$. By Proposition 6.3.3 in [33]
\begin{eqnarray*}
\tilde{T}^{p}_{\mu}(z)&=&\int_{M}<T^{p}_{\mu}s_{z}(w),s_{z}(w)>dv_{g}(w)\\
&\geq &\Bigl(\int_{M}<T_{\mu}s_{z}(w),s_{z}(w)>\Bigr)^{p}\\
&=&(\tilde{\mu}(z))^{p}
\end{eqnarray*}
Hence by Lemma 5.2
$$
\int_{M}(\tilde{\mu}(z))^{p} dv_{g}(z)\leq\int_{M}\vert  \tilde{T^{p}}_{\mu}(z)\vert\leq\hbox{tr}(T_{\mu}^{p}) < \infty
$$
Then $\phi\in L^{p}(M,dv_{g})$.\\
$(b)\Longrightarrow(c)$ Put
$$
\phi_{r}(z)=\mu(B_{g}(z,r))
$$
By diagonal estimates for the Bergman kernel, for some  $\epsilon > 0$ we have
\begin{eqnarray*}
\mu(B_{g}(z,\epsilon))&\preceq &\int_{B_{g}(z,\epsilon)}\vert K(z,w)\vert^{2}d\mu(w)\\
&\preceq & {1\over \vert K(z,z)\vert}\int_{B_{g}(z,\epsilon)}\vert K(z,w)\vert^{2}d\mu(w)\\
&\preceq &\tilde{\mu}(z)
\end{eqnarray*}
Hence $z\rightarrow\phi_{\epsilon}(z):=\mu(B_{g}(z,\epsilon))\in L^{p}(M,dv_{g})$.\\
$(c)\Longrightarrow(a)$.  Suppose that $T_{\phi}\in \mathcal{S}_{p}$. For $z_{0}\in M$ fixed, write $\mu=\mu_{1}+\mu_{2}$ where $$
\mu_{1}:=\mu\mid_{B_{g}(z_{0},\epsilon)}\quad\hbox{and}\quad\mu_{2}:=\mu\mid_{M\setminus B_{g}(z_{0},\epsilon)}
$$
By Corollary 5.5  $T_{\mu_{1}}\in\mathcal{S}_{p}$ . Hence it suffices to show that $T_{\mu_{2}}\in\mathcal{S}_{p}$. If $\sigma\in\mathcal{F}^{2}(M,L)$ we have
\begin{eqnarray*}
<T_{\phi_{\epsilon}}\sigma,\sigma>&=&\int_{M}\vert \sigma(w)\vert\phi_{\epsilon}(w) dv_{g}(w)\\
&=& \int_{M}\vert \sigma(w)\vert^{2}\mu(B_{g}(w,\epsilon)) dv_{g}(w)\\
&\geq &\int_{z\in M}\int_{B_{g}(z,\epsilon)}\vert \sigma(w)\vert^{2}dv_{g}(w)d\mu(z)\\
&\succeq & \int_{M\setminus B_{g}(z_{0},\epsilon)}\vert \sigma(z)\vert^{2} d\mu(z)\quad\hbox{(Prop. 3.1)}\\
&\succeq & <T_{\mu_{2}}\sigma,\sigma>
\end{eqnarray*}
Hence $T_{\mu_{2}} \preceq T_{\phi_{\epsilon}}$ so that $\Vert T_{\mu_{2}}\Vert_{p}\preceq\Vert T_{\phi_{\epsilon}}\Vert_{p}$ and then $T_{\mu_{2}}\in \mathcal{S}_{p}$.
\subsection{Proof of Theorem 1.4 for the case $0 < p < 1$}
By lemma 5.6, it suffices to prove $(a)\Longrightarrow (d)$ and $(b)\Longrightarrow (a)$.\\
$(a)\Longrightarrow (d)$ Suppose that $T_{\mu}\in\mathcal{S}_{p}$. By near diagonal uniform estimate for the Bergman kernel there exists  $\delta > 0$ such that
\begin{equation}
\forall\ z\in M,\ \forall\ w\in B_{g}(z,\delta )\ :\ \vert K(w,z)\vert\succeq 1
\end{equation}
Let $r\geq 2\delta$ and $(a_{j})$ an $r$-lattice. Let $(a_{k_{j}})\subset (a_{j})$ such that $d_{g}(a_{k_{j}},a_{k_{l}}) > r$ if $j\not=l$ so that
\begin{equation}
d_{g}(w,a_{k_{j}})  \leq  r/2\Longrightarrow d_{g}(w,a_{k_{l}})\geq  r/2\\
 \end{equation}
 and
\begin{equation}
d_{g}(w,a_{k_{j}})  \leq  r/2\Longrightarrow d_{g}(w,a_{k_{l}})\geq  {1\over 2}d_{g}(a_{k_{j}},a_{k_{l}})
 \end{equation}
Let $\nu$ be the positive measure
$$
\nu:=\sum_{j}\mathsf{1}_{B_{g}(a_{j},\delta)}\mu
$$
Then $T_{\nu}\leq T_{\mu}$ so that $\Vert T_{\nu}\Vert_{p}\leq \Vert T_{\mu}\Vert_{p}$. Let $(e_{t})$ be  an orthonormal basis of $\mathcal{F}^{2}(M,L)$ and $H : \mathcal{F}^{2}(M,L)\rightarrow \mathcal{F}^{2}(M,L)$ the operator defined by
$$
He_{m}=s_{a_{k_{m}}}
$$
here $s_{a_{k_{m}}}\in\mathcal{F}^{2}(M,L)$ defined as
$$
s_{a_{k_{m}}}(w)\otimes {e(a_{k_{m}})\over\vert e(a_{k_{m}})\vert}=K(w,a_{k_{m}})
$$
where $e$ is a frame of $L$ arround $a_{k_{m}}$. By off-diagonal estimate  for the Bergman kernel
$$
\forall\ w\in M\ :\ \vert s_{a_{k_{m}}}(w)\vert \preceq e^{-\alpha d_{g}(w,a_{k_{m}})}
$$
By lemma 5.4, $H$ extends to a bounded operator on all of $\mathcal{F}^{2}(M,L)$ whose
operator norm is bounded above by a constant that only depends of $(a_{k_{m}})$. If $R=H^{*}T_{\nu}H$
then $$\Vert R\Vert_{p}\leq \Vert T_{\nu}\Vert_{p}\leq \Vert T_{\mu}\Vert_{p}$$
Consider the  operators $\Delta$ and $E$ defined by
$$
\Delta s:=\sum_{m}<He_{m},e_{m}><s,e_{m}>e_{m}\quad\hbox{and}\quad E=R-\Delta
$$
By $(5.1)$ we have
\begin{equation}
{1\over 2}\Vert \Delta\Vert_{p}^{p}-\Vert E\Vert_{p}^{p}\leq\Vert  H\Vert_{p}^{p}\leq\Vert T_{\mu}\Vert_{p}^{p}
\end{equation}
We estimate $\Vert \Delta\Vert_{p}$ from below,
\begin{eqnarray*}
\Vert\Delta\Vert_{p}^{p}&=&\sum_{m}<De_{m},e_{m}>^{p}\\
&=&\sum_{m}<T_{\nu}a_{k_{m}},a_{k_{m}}>^{p}\\
&=&\sum_{m}\Bigl(\int_{M}\vert s_{a_{k_{m}}}(w)\vert^{2}d\nu(w)\Bigr)^{p}\\
&=&\sum_{m}\Bigl(\int_{M}\vert K(w,a_{k_{m}})\vert^{2}d\nu(w)\Bigr)^{p}\\
&\geq &\sum_{m}\Bigl(\int_{B_{g}(a_{k_{m}},\delta)}\vert K(w,a_{k_{m}})\vert^{2}d\nu(w)\Bigr)^{p}\\
&\succeq &\sum_{m}(\mu(B_{g}(a_{k_{m}},\delta))^{p}
\end{eqnarray*}
Thus
\begin{equation}
\Vert\Delta\Vert_{p}^{p} \succeq \sum_{m}(\mu(B_{g}(a_{k_{m}},\delta))^{p}
\end{equation}
We estimate $\Vert E\Vert_{p}$ from above,
\begin{eqnarray*}
\Vert E\Vert_{p}^{p} &\leq & \sum_{l\not=m}<Re_{m},e_{k}>^{p}\\
&= &\sum_{l\not=m}<T_{\nu}e_{m},e_{k}>^{p}\\
&\leq & \sum_{l\not=m}<T_{\nu}s_{a_{k_{m}}},s_{a_{k_{l}}}>^{p}\\
&\leq & \sum_{l\not=m}\Bigl(\int_{M}\vert s_{a_{k_{m}}}(w)\vert\vert s_{a_{k_{l}}}(w)\vert d\nu(w)\Bigr)^{p}\\
&\leq & \sum_{l\not=m}\Bigl(\int_{M} e^{-\alpha d_{g}(w,a_{k_{m}})}e^{-\alpha d_{g}(w,a_{k_{l}})}d\nu(w)\Bigr)^{p}\\
&\preceq & e^{-\alpha pr\over 2}\sum_{m\not=l}\Bigl(\int_{M}e^{-{\alpha\over 2} d_{g}(w,a_{k_{m}})}e^{-{\alpha\over 2} d_{g}(w,a_{k_{l}})}d\nu(w)\Bigr)^{p}\ \hbox{(5.3)}\\
&\preceq & e^{-\alpha pr\over 2}\sum_{m\not=l}\Bigl(\sum_{j}\int_{B_{g}(a_{k_{j}},\delta)}e^{-{\alpha\over 2} d_{g}(w,a_{k_{m}})}e^{-{\alpha\over 2} d_{g}(w,a_{k_{l}})}d\nu(w)\Bigr)^{p}\\
&\preceq & e^{-\alpha pr\over 2}\sum_{m\not=l}\Bigl(\sum_{j}\mu(B_{g}(a_{k_{j}},\delta))e^{-{\alpha\over 4} d_{g}(a_{k_{m}},a_{k_{j}})}e^{-{\alpha\over 4} d_{g}(a_{k_{l}},a_{k_{j}})}d\nu(w)\Bigr)^{p}\ \hbox{(5.4)}
\end{eqnarray*}
Since $0 < p < 1$
\begin{eqnarray*}
\Vert E\Vert_{p}^{p} &\preceq & e^{-\alpha pr\over 2}\sum_{j}\mu(B_{g}(a_{k_{j}},\delta))^{p}\sum_{m\not=k}e^{-{\alpha\over 4} d_{g}(a_{k_{m}},a_{k_{j}})}e^{-{\alpha\over 4} d_{g}(a_{k_{l}},a_{k_{j}})}\\
&\preceq & e^{-\alpha pr\over 2}\sum_{j}\mu(B_{g}(a_{k_{j}},\delta))^{p}\Bigl(\sum_{l}e^{-{\alpha\over 4} d_{g}(a_{k_{l}},a_{k_{j}})}\Bigr)^{2}\\
&\preceq & e^{-\alpha pr\over 2}\sum_{j}\mu(B_{g}(a_{k_{j}},\delta))^{p}
\end{eqnarray*}
Thus
\begin{equation}
\Vert E\Vert_{p}^{p}\preceq e^{-\alpha pr\over 2}\sum_{j}\mu(B_{g}(a_{k_{j}},\delta))^{p}
\end{equation}
By (5.5), (5.6) and (5.7), for $r$ large enough
\begin{eqnarray*}
\Vert T_{\mu}\Vert_{p}^{p} &\geq &\Bigl({c_{1}\over 2}-c_{2}e^{-\alpha pr\over 2}\Bigr)\sum_{j}\mu(B_{g}(a_{k_{j}},\delta))^{p}\\
& \succeq &\sum_{j}\mu(B_{g}(a_{k_{j}},\delta))^{p}
\end{eqnarray*}
for each sub-lattice $(a_{k_{j}})$ of the $r$-lattice $(a_{j})$. Thus
$$
\sum_{j}\mu(B_{g}(a_{j},\delta))^{p} \preceq \Vert T_{\mu}\Vert_{p}^{p}
$$
$(b)\Longrightarrow (a)$ Suppose that $\tilde{\mu}\in L^{p}(M,dv_{g})$. By Lemma 5.6 it suffice to show
$$
\mu(B_{g}(.,\delta))\in L^{p}(Mdv_{g})\Longrightarrow T_{\mu}\in\mathcal{S}_{p}
$$
Let $\phi_{r}(z):=\mu(B_{g}(z,\delta))$. If $s\in\mathcal{F}^{2}(M,L)$ we have
\begin{eqnarray*}
<T_{\phi_{r}}s,s>&=&\int_{M}\vert s(z)\vert^{2}\mu(B_{g}(z,\delta))dv_{g}(z)\\
&=&\int_{M}\vert s(z)\vert^{2}dv_{g}(z)\int_{M}\mathsf{1}_{B_{g}(w,\delta)}d\mu(w)\\
&=&\int_{M}d\mu(w)\int_{M}\vert s(z)\vert^{2}\mathsf{1}_{B_{g}(w,\delta)}dv_{g}(z)\\
&=& \int_{M}d\mu(w)\int_{B_{g}(w,\delta)}\vert s(z)\vert^{2}dv_{g}(z)\\
&\succeq & <T_{\mu}s,s>
\end{eqnarray*}
Thus $T_{\mu} \preceq T_{\phi_{r}}$. Since  $T_{\phi_{r}}\in \mathcal{S}_{p}$ ( Lemma 5.5 ) we get  $T_{\mu}\in\mathcal{S}_{p}$.
\subsection{Proof of Theorem 1.5}
For the proof of Theorem 1.5, we need some preliminary lemmas.\\
 Let $(M,g)$ be a K\"ahler manifold and $(L,h)\rightarrow M$ be a holomorphic hermitian line bundle. Let $(N,\omega_{N})$ be a an Hermitian manifold. For a holomorphic map $\Phi : N\rightarrow M$, let $(\Phi^{*}L,\Phi^{*}h)\rightarrow N$ the holomorpic hermitian line bundle, called  the pull back of $L$,  whose fibers are $(\Phi^{*}L)_{x}=L_{\Phi(x)}$ with metrics $(\Phi^{*}h)(x)=h(\Phi(x))$ where $x\in N$. We define the composition operator
\begin{eqnarray*}
C_{\Phi} :  \mathcal{F}^{2}(M,L)&\longrightarrow &\mathcal{F}^{2}(N,\Phi^{*}L)\\
  s &\longrightarrow & s\circ \Phi
\end{eqnarray*}
The transform $B_{\Phi}$
(related to the usual Berezin transform) associated to $\Phi$ is the function  on $M$ defined as  follows
$$
B_{\Phi}(z)^{2}:=\int_{M}\vert K(z,w)\vert^{2}d\nu_{\Phi}(w)
$$
where $\nu_{\Phi}$ is the pull-back measure defined as follows : for all Borel set $E\subset M$
$$
\nu_{\Phi}(E)=\int_{N}\mathsf{1}_{\Phi^{-1}(E)}(w)dv_{\omega_{N}}(w)
$$
Let $z\in M$. Fix a frame $e$ in a neighborhood $U$ of the point $z$ and consider an orhonormal basis $(s_{j})_{j=1}^{d}$ of $\mathcal{F}^{2}(X,L)$ ( where $1\leq d\leq\infty$). In $U$ each $s_{i}$ is represented by a holomorphic function $f_{i}$ such that $s_{i}(x)=f_{i}(x)e(x)$. Let
$$
s_{z}(w):=\vert e(z)\vert\sum_{i=1}^{d}\overline{f_{i}(z)}s_{i}(w)
$$
Then  $s_{z}$ is a holomorphic section and
\begin{eqnarray*}
\vert s_{z}(w)\vert&=&\Big\vert\Bigl(\sum_{i=1}^{d}\overline{f_{i}(z)}s_{i}(w)\Bigr)\otimes\overline{e(z)}\Big\vert\\
&=&\Big\vert\sum_{i=1}^{d}s_{i}(w)\otimes\overline{s_{i}(z)}\Big\vert\\
&=&\vert K(w,z)\vert
\end{eqnarray*}
By proposition 3.3
\begin{eqnarray*}
\int_{M}\vert s_{z}\vert^{2}dv_{g}(w)&=&\int_{M}\vert K(w,z)\vert^{2}dv_{g}(w)\\
&=&\vert K(z,z)\vert\asymp 1
\end{eqnarray*}
\begin{lem}
We have
$$
<C_{\Phi}^{*}C_{\Phi}s_{z},s_{z}>=B_{\Phi}(z)^{2}
$$
$$
B_{\Phi}(z)^{2}=\int_{M}\vert s_{z}(w)\vert^{2}d\nu_{\Phi}(w)
$$
and
$$
\int_{M}\vert B_{\Phi}(z)\vert^{p}dv_{g}(z)=\int_{M}<C_{\Phi}^{*}C_{\Phi}s_{z},s_{z}>^{p\over 2}dv_{g}(z)
$$
where $\nu_{\Phi}$ is the pull-back measure defined as follows : for all Borel set $E\subset M$
$$
\nu_{\Phi}(E)=\int_{N}\mathsf{1}_{\Phi^{-1}(E)}(w)dv_{\omega_{N}}(w)
$$
\end{lem}
\begin{proof}
We have
\begin{eqnarray*}
<C_{\Phi}^{*}C_{\Phi}s_{z},s_{z}>&=&<C_{\Phi}s_{z},C_{\Phi}s_{z}>\\
&=&\int_{N}\vert s_{z}(\Phi(w))\vert^{2} dv_{\omega_{N}}(w)\\
&=&\int_{M}\vert s_{z}(w)\vert^{2}d\nu_{\Phi}(w)\\
&=&\int_{M}\vert K(z,w)\vert^{2}d\nu_{\Phi}(w)\\
&=& \int_{M}\vert K(z,\Phi(w))\vert^{2}dv_{g}(w)\\
&=& B_{\Phi}(z)^{2}
\end{eqnarray*}
\end{proof}
\noindent The following lemma presents a desired connection between  composition
operators and Toeplitz operators.
\begin{lem}
Let $(M,g)$ be a K\"ahler manifold and let $\Phi : N\rightarrow M$ be a holomorphic map such that $C_{\Phi}$ is bounded. Then
$$
C_{\Phi}^{*}C_{\Phi}=T_{\nu_{\Phi}}
$$
where
$$
T_{\nu_{\Phi}}s(z)=\int_{M}<s(w),K(w,z)>d\nu_{\Phi}(w)
$$
\end{lem}
\begin{proof}
Since $C_{\Phi}$ is bounded,   for all $s,\sigma\in\mathcal{F}^{2}(M,L)$
\begin{eqnarray*}
<C_{\Phi}^{*}C_{\Phi}s,\sigma> &=& <C_{\Phi}s,C_{\Phi}\sigma>\\
&=&\int_{N}<s(\Phi(w)),\sigma(\Phi(w))>dv_{\omega_{N}}(w)\\
&=&\int_{M}<s(w),\sigma(w)>d\nu_{\Phi}(w)
\end{eqnarray*}
Since
$$
\sigma(w)=\int_{M}<\sigma(z),K(z,w)>dv_{g}(z)
$$
By Fubini Theorem
\begin{eqnarray*}
<C_{\Phi}^{*}C_{\Phi}s,\sigma>&=&\int_{M}<s(w),\int_{M}K(w,z).\sigma(t)dv_{g}(z)>d\nu_{\Phi}(w)\\
&=&\int_{M}\int_{M}<s(w),K(w,z).\sigma(t)>dv_{g}(z)d\nu_{\Phi}(w)\\
&=&\int_{M}\int_{M}<K(z,w).s(w),\sigma(t)> dv_{g}(z)d\nu_{\Phi}(w)\\
&=&\int_{M}<\int_{M}K(z,w).s(w)d\nu_{\Phi}(w),\sigma(z)>dv_{g}(z)\\
&=&<\int_{M}K(.,w).s(w)d\nu_{\Phi}(w),\sigma>
\end{eqnarray*}
we get
$$
C_{\Phi}^{*}C_{\Phi}s(z)=\int_{M}<s(w),K(w,z)>d\nu_{\Phi}(w)
$$
\end{proof}
\begin{cor}
Let $(M,g)$ be a K\"ahler manifold and let $\Phi : N\rightarrow M$ be a holomorphic map such that $C_{\Phi} : \mathcal{F}^{2}(M,L)\rightarrow\mathcal{F}^{2}(N,\Phi^{*}L)$ is bounded. If $0 < p < \infty$, then
$
C_{\Phi}\in\mathcal{S}_{p}\quad\hbox{if and only if}\quad T_{\nu_{\Phi}}\in\mathcal{S}_{p/2}$.
\end{cor}
\noindent Since  $\vert K(z,z)\vert \asymp 1$ and
\begin{eqnarray*}
\tilde{\nu}_{\Phi}(z)&=&{1\over\vert K(z,z)\vert}\int_{M}\vert K(z,w)\vert^{2}d\nu_{\Phi}(w)
\asymp \int_{M}\vert K(z,w)\vert^{2}d\nu_{\Phi}(w)\\
&\asymp & B_{\Phi}(z)^{2}
\end{eqnarray*}
then the proof of Theorem 1.5 follows from  Theorems 1.2;1.3 and 1.4.


\end{document}